\documentclass{jmsj2}
\usepackage{amssymb}
\usepackage{mathtools}
\theoremstyle{plain}
\newtheorem{theorem}{Theorem}[section]
\newtheorem{lemma}[theorem]{Lemma}
\newtheorem{cor}[theorem]{Corollary}
\newtheorem{prop}[theorem]{Proposition}
\theoremstyle{definition}
\newtheorem{definition}[theorem]{Definition}

\newtheorem{example}[theorem]{Example}
\theoremstyle{remark}
\newtheorem{rem}[theorem]{Remark}
\newtheorem*{notation}{Notation}
\DeclareMathOperator{\bin}{Bin}
\DeclareMathOperator{\cl}{Cl}
\DeclareMathOperator{\im}{Im}
\DeclareMathOperator{\Li}{Li}
\DeclareMathOperator{\lk}{lk}
\DeclareMathOperator{\ph}{PH}
\DeclareMathOperator{\rank}{rank}
\DeclareMathOperator{\sgn}{sgn}
\DeclareMathOperator{\spa}{span}
\DeclareMathOperator{\tr}{tr}

\newcommand{\qad}{\phantom{={}}}
\newcommand{\relmiddle}[1]{\mathrel{}\middle#1\mathrel{}}
\renewcommand{\P}{\mathbb{P}}
\newcommand{\E}{\mathbb{E}}
\newcommand{\N}{\mathbb{N}}
\newcommand{\R}{\mathbb{R}}
\newcommand{\Z}{\mathbb{Z}}
\newcommand{\cC}{\mathcal{C}}
\newcommand{\cF}{\mathcal{F}}
\newcommand{\cK}{\mathcal{K}}
\newcommand{\cL}{\mathcal{L}}
\newcommand{\cX}{\mathcal{X}}
\newcommand{\p}{\mathbf{p}}
\renewcommand{\a}{\alpha}
\renewcommand{\b}{\beta}
\newcommand{\dl}{\delta}
\newcommand{\eps}{\varepsilon}
\newcommand{\gm}{\gamma}
\newcommand{\lm}{\lambda}
\newcommand{\Om}{\Omega}
\newcommand{\sg}{\sigma}
\newcommand{\Sg}{\Sigma}
\newcommand{\la}{\langle}
\newcommand{\ra}{\rangle}
\newcommand{\del}{\partial}
\numberwithin{equation}{section}
\DeclareRobustCommand{\qed}{%
  \ifmmode \mathqed
  \else
    \leavevmode\unskip\penalty9999 \hbox{}\nobreak\hfill
    \quad\hbox{\qedsymbol}%
  \fi
}
\newcommand{\myqedhere}{\qedhere}

\allowdisplaybreaks[3]
\begin{document}

\title[Lifetime sums for random simplicial complex processes]{Asymptotic behavior of lifetime sums for random simplicial complex processes}

\author{Masanori \textsc{Hino}}
\address{Department of Mathematics\\ Kyoto University\\ Kyoto 606--8502, Japan}
\email{hino@math.kyoto-u.ac.jp}

\author{Shu \textsc{Kanazawa}}
\address{Mathematics Department\\ Tohoku University\\ Sendai 980--8578, Japan}
\email{kanazawa.shu.p5@dc.tohoku.ac.jp}

\subjclass[2010]{Primary 05C80, 60D05; Secondary 55U10, 05E45, 60C05}

\keywords{Linial--Meshulam complex process, random clique complex process, multi-parameter random simplicial complex, lifetime sum, Betti number}


\begin{abstract}
We study the homological properties of random simplicial complexes. In particular, we obtain the asymptotic behavior of lifetime sums for a class of increasing random simplicial complexes; this result is a higher-dimensional counterpart of Frieze's $\zeta(3)$-limit theorem for the Erd\H{o}s--R\'{e}nyi graph process. The main results include solutions to questions posed in an earlier study by Hiraoka and Shirai about the Linial--Meshulam complex process and the random clique complex process. One of the key elements of the arguments is a new upper bound on the Betti numbers of general simplicial complexes in terms of the number of small eigenvalues of Laplacians on links. This bound can be regarded as a quantitative version of the cohomology vanishing theorem.
\end{abstract}

\maketitle

\section{Introduction}
The Erd\H{o}s--R\'{e}nyi $G(n,p)$ model has been extensively studied since the 1960s~(\cite{Gi,ER1,ER2}). This model, defined as the distribution of random graphs with $n$ vertices where the edge between each pair of vertices is included with probability $p$ independently of any other edge, is one of the most typical models of random graphs.
One of the main themes in $G(n,p)$ theory is searching for threshold probabilities. For example, Erd\H{o}s and R\'{e}nyi~\cite{ER2} showed that the threshold for graph connectivity of $G(n,p)$ is $p = (\log n)/n$. When we vary $p$ and consider a family of the Erd\H{o}s--R\'{e}nyi graphs with parameter $p$, the following construction is often useful. Let $K_n = V_n \sqcup E_n$ be the complete graph with $n$ given vertices, where $V_n$ and $E_n$ denote the sets of vertices and edges, respectively. We assign an independent random variable $u_e$ to each edge $e\in E_n$ and let $u_e$ be uniformly random on $[0,1]$. For each $p\in[0,1]$, a random subgraph $K_n(p)$ of $K_n$ is then defined by
\[
K_n(p) := V_n \sqcup \{e\in E_n\mid u_e\le p\}. 
\]
This construction, the so-called Erd\H{o}s--R\'{e}nyi graph process over $n$ vertices, yields an increasing family $\cK_n := \{K_n(t)\}_{t\in [0, 1]}$ of random graphs. 
This process is closely related to the concept of the minimum weight on $K_n$, which can be seen as follows. For each spanning tree $T$ in $K_n$, define its weight as $\sum_{e\in T} u_e$. 
Let $W_n$ be the minimum weight among all the spanning trees in $K_n$. Then
\begin{equation}\label{eq: key of Frieze's zeta 3}
W_n = \sum_{i=1}^{n-1}t_i = \int_0^1\b_0(K_n(t))\,dt, 
\end{equation}
where $t_i\in [0, 1]$ is the $i$-th random time at which the number of connected components of $K_n(t)$ decreases, and $\b_0(K_n(t))$ denotes the zeroth (reduced) Betti number of $K_n(t)$, that is, the number of connected components of $K_n(t)$ minus one. 
This type of relation holds for a general increasing family of graphs.
Applying this formula and analyzing $\b_0(K_n(t))$ in detail, Frieze~\cite{F} obtained the following significant result about the behavior of $W_n$. 
\begin{theorem}[$\zeta(3)$-limit theorem {\cite{F}}]\label{thm:FriezeZeta}
It holds that
\[
\lim_{n\to\infty}\E[W_n]=\zeta(3)\biggl(=\sum_{k=1}^{\infty}k^{-3}=1.202\cdots\biggr)
\]
and for any $\eps>0$, 
\[
\lim_{n\to\infty}\P(|W_n-\zeta(3)|>\eps)=0. 
\]
\end{theorem}

Recently, there has been a growing interest in studying random simplicial complexes as a higher-dimensional generalization of random graphs. Since an Erd\H{o}s--R\'{e}nyi graph can be regarded as a one-dimensional random simplicial complex, and graph connectivity can be equivalently described as the vanishing of the zeroth (reduced) homology, it is natural to seek a higher-dimensional analogue to the theory of Erd\H{o}s--R\'{e}nyi's $G(n,p)$ model. 
The $d$-Linial--Meshulam model~\cite{LM1} and the random clique complex model~\cite{K2} are typical models of this type. 
The $d$-Linial--Meshulam model $Y_d(n,p)$ is defined as the distribution of $d$-dimensional random simplicial complexes with $n$ vertices and the complete $(d-1)$-dimensional skeleton such that each $d$-simplex is placed with independent probability $p$. 
The random clique complex model $C(n,p)$ is defined as the distribution of the clique complex of the Erd\H{o}s--R\'{e}nyi graph that follows $G(n,p)$. Here, given a graph $G$, its clique complex $\cl(G)$ is defined as the maximal simplicial complex among those for which the one-dimensional skeletons are equal to $G$. 
Linial, Meshulam, and Wallach~\cite{LM1,LM2} exhibited the threshold for the vanishing of the $(d-1)$-th homology for the $d$-Linial--Meshulam model, which is analogous to the connectivity threshold of the Erd\H{o}s--R\'{e}nyi graph. Later, Kahle~\cite{K1} obtained similar results for the random clique complex model.

Along another line, Hiraoka and Shirai~\cite{HS} obtained a higher-dimensional analogue of \eqref{eq: key of Frieze's zeta 3} in the context of the 
theory of persistent homology. 
Persistent homologies can describe the topological features of a filtration (i.e., an increasing family of simplicial complexes; see, e.g.,~\cite{ELZ,ZC}). In particular, these provide rigorous definitions for the concepts of birth and death times of higher-dimensional holes, sometimes called cycles and cavities. The lifetimes, which are defined as the difference between the birth time and death time, measure the persistence of each hole in the filtration. 
In \cite{HS}, Hiraoka and Shirai proved that the lifetime sum $L_k(\cX)$ of the $k$-th persistent homology associated with a filtration 
$\cX=\{X(t)\}_{0\le t\le1}$ is equal to $\int_0^1 \b_k(X(t))\,dt$, where $\b_k(X(t))$ represents the $k$-th (reduced) Betti number of the simplicial complex $X(t)$ (see Theorem~\ref{thm:lifetimeformula}). When $k=0$ and $\cX$ is the Erd\H{o}s--R\'{e}nyi graph process over $n$ vertices, the result is consistent with the second identity of \eqref{eq: key of Frieze's zeta 3}. 
They also obtained a relation analogous to the first identity (see Theorem~1.1 in~\cite{HS}).
Thus, it is natural to seek the asymptotic behavior of the lifetime sum $L_k(\cX_n)$ for a random filtration $\cX_n=\{X_n(t)\}_{0\le t\le1}$ or $\cX_n=\{X_n(t)\}_{t\ge0}$ over $n$ vertices as a higher-dimensional generalization of Theorem~\ref{thm:FriezeZeta}. 
Typical random filtrations include the $d$-Linial--Meshulam complex process $\cK_n^{(d)}=\{K_n^{(d)}(t)\}_{0\le t\le1}$, where $K_n^{(d)}(t)$ follows $Y_d(n,t)$, and the random clique complex process $\cC_n = \left\{C_n(t)\right\}_{0\le t\le1}$, where $C_n(t)$ follows $C(n,t)$ (see Section~4.2 for further details).
For these models, the following estimates were proved by Hiraoka and Shirai.
\begin{theorem}[{\cite[Theorem~1.2]{HS}}]\label{thm:HS1}
Let $d\ge 1$. 
There exist positive constants $c$ and $C$ such that for sufficiently large $n$, 
\[
cn^{d-1}\le \E\bigl[L_{d-1}(\cK_n^{(d)})\bigr]\le C n^{d-1}. 
\]
\end{theorem}
\begin{theorem}	[{\cite[Theorem~6.10]{HS}}]\label{thm:HS2}
Let $k\ge 0$.
There exist positive constants $c$ and $C$ such that for sufficiently large $n$, 
\begin{equation}\label{eq: Hiraoka--Shirai 2}
c n^{k/2+1-1/(k+1)} \le \E[L_k(\cC_n)] \le \begin{cases}
C n^k\log n 	&(k=0, 1), \\
C n^k		&(k\ge 2). 
\end{cases}
\end{equation}
\end{theorem}
In \cite{HS}, the asymptotic behavior of $\E[L_{d-1}(\cK_n^{(d)})]/n^{d-1}$ as $n\to\infty$ is also discussed, and a possible limiting constant $I_{d-1}$ is found by a heuristic argument.
For the exact value of $I_{d-1}$, see \eqref{eq:hdc} and \eqref{eq:I} and Section~\ref{sec:I}.
The exact growth exponent of $ \E[L_k(\cC_n)]$, mentioned as a problem in that paper, is considered here. 

In this paper, we obtain sharp quantitative estimates of lifetime sums $L_k(\cX_n)$ for a large class of random filtrations over $n$ vertices. 
The main results include a rigorous proof of the convergence of $L_k(\cK_n^{(d)})/n^{d-1}$ to $I_{d-1}$ and a determination of the growth exponent of $ \E[L_k(\cC_n)]$. These results solve problems posed in~\cite{HS} and are summarized in the following theorems.
\begin{theorem}	\label{thm:LMlimiting}
Let $d\ge1$. For any $r \in[1,\infty)$,
\[
\lim_{n\to\infty}\E\left[\left|\frac{L_{d-1}(\cK_n^{(d)})}{n^{d-1}} - I_{d-1}\right|^r\right]=0.
\]
In particular, $\E[L_{d-1}(\cK_n^{(d)})]/n^{d-1}$ converges to $I_{d-1}$ as $n\to\infty$.
\end{theorem}
\begin{theorem}\label{thm: order of 1-flag}
Let $k\ge 0$. There exist positive constants $c$ and $C$ such that, for sufficiently large $n$, 
\[
c n^{k/2+1-1/(k+1)} \le \E[L_k(\cC_n)] \le C n^{k/2+1-1/(k+1)}. 
\]
\end{theorem}
As seen from Theorem~\ref{thm: order of 1-flag}, the exponent of the lower estimate in \eqref{eq: Hiraoka--Shirai 2} is exact.
To prove these theorems, we introduce a new upper estimate of the Betti numbers of general simplicial complexes (Theorem~\ref{thm:geneCVT}). This estimate is a quantitative version of the cohomology vanishing theorem~(\cite{G,BS}, see also Theorem~\ref{thm: cohomology vanishing theorem}). By applying this theorem and modified versions of known estimates to a family of multi-parameter random simplicial complexes, including the $d$-Linial--Meshulam complex and the random clique complex, we obtain several inequalities involving the expectations of Betti numbers. Theorem~\ref{thm:CFdecayBetti}, in particular, provides an essentially new upper estimate. By integrating these inequalities with respect to the filtration parameter, we obtain good estimates of 
$\E[L_k(\cX_n)]$ for a class of random filtrations over $n$ vertices (Theorems~\ref{thm:CForder} and \ref{thm:CForder2}). 
Theorem~\ref{thm: order of 1-flag} is a special case of these theorems, and Theorem~\ref{thm:HS1} also follows from them.
The proof of Theorem~\ref{thm:HS1} in \cite{HS} is based on the monotonicity of $\b_{d-1}(K_n^{(d)}(t))$ with respect to $t$. Our approach is different and is applicable to more general random filtrations. For the proof of Theorem~\ref{thm:LMlimiting} and an extension (Theorem~\ref{thm:LMlimiting2}), we additionally make use of the result by Linial and Peled~\cite{LP} on the convergence of $\b_d(K^{(d)}_n(c/n))/n^d$ as $n\to\infty$ for each $c\ge0$.

This paper is organized as follows. In Section~2, we provide fundamental concepts for graphs and simplicial complexes and derive upper estimates of Betti numbers as a quantitative generalization of the cohomology vanishing theorem. In Section~3, we provide several estimates for the expectations of Betti numbers for a class of random simplicial complexes. In Section~4, we introduce the concept of persistent homologies and prove the main theorems about lifetime sums for random simplicial complex processes. 

\begin{notation}
We use the Bachmann--Landau big-$O$, little-$o$, and some related notation associated with $n$ (the number of vertices) tending to $\infty$. Furthermore, for non-negative functions $f(n)$ and $g(n)$, 
\begin{itemize}
\item $f(n)=\Om(g(n))$ means that $g(n)=O(f(n))$; and
\item $f(n) \asymp g(n)$ means that $f(n) = O(g(n))$ and $g(n) = O(f(n))$.
\end{itemize}
The notation $X\sim \nu$ indicates that a random variable $X$ has probability distribution $\nu$.
For $a,b\in\R$, $a\vee b$ and $a\wedge b$ denote $\max\{a,b\}$ and $\min\{a,b\}$, respectively.
\end{notation}
\section{Upper bounds of Betti numbers of simplicial complexes}
\subsection{Preliminaries and statement of results}
Let $V$ be a finite set.
For $k\ge0$, $\binom{V}{k}$ denotes the set of all subsets $A$ of $V$ whose cardinalities $\#A$ are $k$. Note that $\binom{V}{0}$ contains a single element $\emptyset$. Assume $V\ne\emptyset$ and let $E\subset \binom{V}{2}$. We regard $V$ and $E$ as a vertex set and an edge set, respectively, and call $G = V \sqcup E$ an undirected graph on $V$. Throughout this article, graphs are simple undirected finite graphs, with no multiple edges and no self-loops.

Saying that $v\in V$ is adjacent to $w\in V$ means $\{v,w\}\in E$.
For $v\in V$, the degree of $v$ is defined as $\#\{w\in V\mid \{v,w\}\in E\}$ and is denoted by $\deg(v)$.
A vertex $v\in V$ is called isolated if $\deg(v)=0$.
The averaging matrix $A[G]=\{a_{vw}\}_{v, w\in V}$ associated with $G$ is defined by
\begin{align*}
a_{vw} &:= \begin{cases}
1/\deg(v)	&\text{if  $w$ is adjacent to $v$,}\\
1				&\text{if $v$ is isolated and $v = w$,}\\
0				&\text{otherwise}.
\end{cases}
\end{align*}
The Laplacian $\cL[G]$ of a simple random walk on $G$ is defined by 
$\cL[G] := I_{V} - A[G]$, where $I_{V}$ is the matrix that acts as the identity operator on $V$. 
All the eigenvalues of $\cL[G]$ are real and belong to the interval $[0, 2]$ (see Section~2.2 for details). Note that the nonzero vectors that are constant on each connected component are eigenvectors with associated eigenvalue zero, and the number of connected components of $G$ coincides with the multiplicity of the zero eigenvalues. When $\#V\ge2$, $\lm_2[G]$ denotes the second smallest eigenvalue of $\cL[G]$, counting multiplicities. In particular, $\lm_2[G]>0$ if and only if $G$ is connected. We call $\lm_2[G]$ the spectral gap of $G$. 
By convention, $\lm_2[\emptyset] = 0$, and we set $\lm_2[G] = 0$ for $\#V=1$.

We next introduce the concept of simplicial complexes, which are higher-dimensional counterparts of graphs.
\begin{definition}
Let $V$ be a nonempty finite set and $X$ a collection of nonempty subsets of $V$. 
$X$ is called an abstract simplicial complex on $V$ if $X$ satisfies the following two conditions. 
\begin{enumerate}
\item $\{v\} \in X$ for all $v\in V$;
\item If $\sg\in X$ and $\emptyset \neq \tau \subset \sg$, then $\tau\in X$.
\end{enumerate}
\end{definition}
In what follows, we omit the word ``abstract'' and simply call $X$ a simplicial complex.
For $\sg\in X$, its dimension $\dim \sg$ is defined to be $\#\sg-1$.
We call $\sg\in X$ with $\dim\sg = k$ a $k$-dimensional simplex or, equivalently, a $k$-simplex. The dimension of $X$ is defined as the maximum among the dimensions of the simplices in $X$. Graphs are regarded as zero- or one-dimensional simplicial complexes in a natural manner.
We say that $\tau\in X$ is a face of $\sg\in X$ whenever $\tau\subset\sigma$.
For $k\ge0$, $X_k$ denotes the set of all $k$-simplices of $X$. By convention, we regard $\emptyset$ as a $(-1)$-simplex and set $X_{-1}=\{\emptyset\}$. 
The $k$-dimensional skeleton $X^{(k)}$ of $X$ is defined by $X^{(k)} := \bigsqcup_{j=0}^k X_j$. The simplicial complex $X$ is said to include the complete $k$-dimensional skeleton if $X_k = \binom{V}{k+1}$. 

Given a simplicial complex $X$ on $V$ and $k\ge0$, an ordered sequence $(v_0, v_1, \ldots, v_k)$ consisting of $k+1$ distinct elements of $V$ is called an ordered ($k$-)simplex of $X$ if $ \{v_0, v_1, \ldots, v_k\} \in X_k$.
The collection of all ordered $k$-simplices of $X$ is denoted by $\Sg(X_k)$, with $\Sg(X) := \bigsqcup_j\Sg(X_j)$. 
By convention, we set $\Sg(X_{-1})=\Sg(\{\emptyset\}):=\{\emptyset\}$. 
Two ordered simplices are called equivalent if they can be transformed into each other by an even permutation. The equivalence class of an ordered simplex $\sg = (v_0, v_1, \ldots, v_k)$ is denoted by $\la\sigma\ra$ or $\la v_0, v_1, \ldots, v_k\ra$ and is called the oriented simplex generated by $\sg$. Let $C_k(X)$ be the $\R$-vector space of all linear combinations of oriented $k$-simplices in $X$ with coefficients in $\R$ under the relation that $\la v_0, v_1, \ldots, v_k\ra = -\la v_1, v_0, \ldots, v_k\ra$ for any oriented $k$-simplices. 
We set $C_{-1}(X) = \R$ per convention. For $k \ge 1$, the $k$-th boundary map $\del_k\colon C_k(X)\to C_{k-1}(X)$ is well-defined as a linear extension of
\[
\del_k\la v_0, v_1, \ldots, v_k\ra := \sum_{i=0}^k(-1)^i\la v_0, \ldots, v_{i-1}, v_{i+1}, \ldots,  v_k\ra
\]
for $\la v_0, v_1, \ldots, v_k\ra\in C_k(X)$. We also define a linear map $\del_0\colon C_0(X)\to\R$ such that $\del_0\la v\ra = 1$ for $v\in V$. For all $k\ge 0$, it holds that $\del_k\circ\partial_{k+1} = 0$, that is, $\ker\partial_k\supset \im\partial_{k+1}$. The $k$-th homology vector space of $X$ over $\R$ is defined by $H_k(X) := \ker\partial_k / \im\partial_{k+1}$. The dimension of this space is called the $k$-th Betti number of $X$ and is denoted by $\b_k(X)$. 
\begin{rem}
In the usual definitions for homologies, $\del_0$ is defined as a zero operator. This difference makes the Betti number $\b_0(X)$ as defined above smaller by 1 than the conventional Betti number. In this sense, $\b_k(X)$ in our definition is often called the \emph{reduced} Betti number. For simplicity, we omit the word ``reduced'' throughout this paper.
\end{rem}
Let $k\ge0$.
A real-valued function $f$ on $\Sg(X_k)$ is called a $k$-cochain if $f$ is alternating, that is, if $f(v_{\xi(0)},v_{\xi(1)},\dots,v_{\xi(k)})=(\sgn \xi)\,f(v_0,v_1,\dots,v_k)$ for all $(v_0,v_1,\dots,v_k)\in\Sg(X_k)$ and all permutations $\xi$ on $\{0,1,\dots,k\}$. The real vector space $C^k(X)$ formed by all $k$-cochains is called the $k$-cochain vector space.  We set $C^{-1}(X) = \R$ per convention. The $k$-th coboundary map $d_k\colon C^k(X)\to C^{k+1}(X)$ is defined by the linear extension of
\[
d_k\varphi(\sg) := \sum_{i=0}^{k+1} (-1)^i\varphi(\sg_i)
\]
for $\varphi\in C^k(X)$ and $\sg = (v_0, \ldots, v_{k+1})\in \Sg(X_{k+1})$, where 
\begin{equation}\label{eq:sg_i}
\sg_i := (v_0, \ldots, v_{i-1}, v_{i+1}, \ldots, v_{k+1})\in\Sg(X_k).
\end{equation}
By definition, $d_{-1}\varphi$ for $\varphi\in C^{-1}(X)=\R$ is identically $\varphi$ on $\Sg(X_0)$. 
Elements of $\im d_{k-1}$ (resp., $\ker d_k$) are called $k$-coboundaries ($k$-cocycles). That $d_k\circ d_{k-1} = 0$, that is, $\ker d_k \supset \im d_{k-1}$, can be verified. The $k$-th cohomology vector space of $X$ is then defined by
$H^k(X) := {\ker d_k}/{\im d_{k-1}}$. 
Note that $H^k(X)$ is isomorphic to $H_k(X)$. 

To state the cohomology vanishing theorem and its quantitative generalization, we further introduce the concept of links of simplicial complexes. Given a $D$-dimensional simplicial complex $X$ and a $j$-simplex 
$\tau$ in $X$ with $-1 \le j \le D$, we define the link $\lk_X(\tau)$ of $\tau$ in $X$ by
\[
\lk_X(\tau) := \{\sg \in X\mid \tau\cap\sg=\emptyset\text{ and }\tau\cup\sigma \in X\}. 
\]
Note that $\lk_X(\tau)$ is either the empty set or a simplicial complex with dimension at most $D - j -1$. If the codimension of $\tau$ in $X$ is no more than $2$ (i.e., $D - j \le 2$), then $\lk_X(\tau)$ is either the empty set or a graph. By definition, $\lk_X(\emptyset)$ is always equal to $X$. 

A simplicial complex $X$ is said to be pure $D$-dimensional if, for every simplex $\sg$ in $X$, there exists some $D$-simplex containing $\sg$. 
Note that a pure $D$-dimensional simplicial complex  is $D$-dimensional. 
The following is a special case of Theorem 2.1 of \cite{BS}.
\begin{theorem}[Cohomology vanishing theorem {\cite{G,BS}}]		\label{thm: cohomology vanishing theorem}
Let $D\ge 1$, and let $X$ be a pure $D$-dimensional simplicial complex such that $\lm_2[\lk_X(\tau)] > 1 - D^{-1}$ for every $(D-2)$-simplex $\tau\in X$.
Then, $H^{D-1}(X) = \{0\}$. 
\end{theorem}
The main purpose of this section is to generalize Theorem~\ref{thm: cohomology vanishing theorem} to an upper estimate of the Betti number. This estimate is one of the key elements of the arguments in the later sections.
\begin{definition}
Let $G$ be a graph on $V$ and let $\{\lm_i\}_{i=1}^{\#V}$ be all the not necessarily distinct  eigenvalues of $\cL[G]$. We define
\[
\gm(G;\a) := \#\left\{i\relmiddle|\lm_i\le\a\right\} -1
\]
for $\a\ge0$. We also set $\gm(\emptyset;\a) := 0$ per convention.
\end{definition}
\begin{theorem}	\label{thm:geneCVT}
Let $X$ be a simplicial complex and $D\ge1$. Then, 
\begin{equation}
\b_{D-1}(X) \le \sum_{\tau \in X_{D-2}} \gm\bigl(\lk_X(\tau)^{(1)};1-D^{-1}\bigr). 	\label{eq:geneCVT}
\end{equation}
\end{theorem}
Recall that the graph $\lk_X(\tau)^{(1)}$ is the one-dimensional skeleton of the simplicial complex $\lk_X(\tau)$.
This theorem is an extension of Theorem~\ref{thm: cohomology vanishing theorem}. In fact, under the assumptions of Theorem~\ref{thm: cohomology vanishing theorem}, $\gm\bigl(\lk_X(\tau)^{(1)};1-D^{-1}\bigr)=\gm(\lk_X(\tau);1-D^{-1})=0$ for every $\tau\in X_{D-2}$, so that $\b_{D-1}(X)=0$.

We devote the rest of this section to proving Theorem~\ref{thm:geneCVT}.
The proof is based on careful modifications to the proof of Theorem~\ref{thm: cohomology vanishing theorem} and a nice transformation of $X$ to remove the assumption of pure dimensionality.
\subsection{Auxiliary operators}
Here, we give an overview of some preliminary concepts and facts about simplicial complexes as described in \cite{BS} and state them in a way that is useful for this research.
Let $X$ be a pure $D$-dimensional simplicial complex. For simplices $\sg$ in $X$, $m(\sg)$ denotes the number of $D$-simplices containing $\sg$.
Note that $m(\cdot) \ge 1$ from the assumption of pure $D$-dimensionality. 
For ordered simplices $\sg=(v_0,\dots,v_k)\in\Sg(X)$, $m(\sg)$ is defined as $m(\{v_0,\dots,v_k\})$. Simple calculations show that, for all $-1 \le k \le D$ and $\tau\in \Sg(X_k)$, 
\[
\sum_{\sigma\in\Sg(X_{k+1});\, \sigma\supset\tau}m(\sg) = (k+2)!\,(D-k)m(\tau), 
\]
where $\sigma\supset\tau$ implies that all vertices of $\tau$ are vertices of $\sg$. 
We equip $C^k(X)$ with an inner product $(\cdot,\cdot)$ defined by 
\begin{align}
(\varphi, \psi)&=\frac{1}{(k+1)!}\sum_{\sg\in\Sg(X_k)}m(\sg)\varphi(\sg)\psi(\sg)
\quad\text{for }\varphi, \psi\in C^k(X),\ k\ge0
 \label{eq:naiseki1}\\
\shortintertext{and}
(\varphi,\psi)&=\#X_D\varphi\psi\quad\text{for }\varphi, \psi\in C^{-1}(X)=\R. \label{eq:naiseki2}
\end{align}
The induced norm is denoted by $\|\cdot\|$. For $k\ge-1$, denote by $\dl_{k+1}\colon C^{k+1}(X)\to C^k(X)$ the adjoint operator of $d_k$, that is, the unique operator satisfying $(d_k\varphi, \psi) = (\varphi, \dl_{k+1}\psi)$ for all $\varphi\in C^k(X)$ and $\psi\in C^{k+1}(X)$. For  ordered simplices $\sg=(v_0,\ldots,v_k)\in\Sg(X_k)$ and $\tau=(w_0,\ldots,w_l)\in\Sg(X_l)$, the notation $\sg\tau$ indicates an ordered sequence $(v_0,\ldots,v_k,w_0,\ldots,w_l)$. A straightforward calculation gives the following expressions.
For $\varphi\in C^k(X)$, $\psi\in C^{k+1}(X)$, and $\sg\in\Sg(X_k)$ with $k\ge0$,
\begin{align}
\dl_{k+1}\psi(\sg) &= \sum_{\substack{v\in\Sg(X_0);\\v\sigma\in\Sg(X_{k+1})}} \frac{m(v\sigma)}{m(\sg)}\psi(v\sigma)\nonumber\\
\shortintertext{and}
\dl_{k+1} d_k\varphi(\sg) &= (D-k)\varphi(\sg) - \sum_{\substack{v\in\Sg(X_0);\\v\sigma\in\Sg(X_{k+1})}}\sum_{i=0}^k(-1)^i\frac{m(v\sigma)}{m(\sg)}\varphi(v\sigma_i).\label{eq:upLap}
\end{align}
See \eqref{eq:sg_i} for the definition of $\sg_i$.

For $k\ge 0$, the down Laplacian and the up Laplacian on $C^k(X)$ are defined by
$
L_k^{\text{down}} := d_{k-1}\dl_k$ and $L_k^{\text{up}} := \dl_{k+1}d_k
$, respectively.
The Laplacian $L_k$ on $C^k(X)$ is defined by 
\[
L_k := L_k^{\text{down}} + L_k^{\text{up}}. 
\]
Note that $L_k^{\text{down}}$, $L_k^{\text{up}}$, and $L_k$ are self-adjoint and non-negative definite operators with respect to the inner product \eqref{eq:naiseki1} and, further, that the relations
\begin{align}
\ker L_k^{\text{down}} &= \ker \dl_k, \label{eq:kerdel}\\
\ker L_k^{\text{up}} &= \ker d_k, \label{eq:kerd}\\
H^k(X)&\simeq \ker L_k^{\text{down}} \cap \ker L_k^{\text{up}}=\ker L_k \label{eq:kerLap}
\end{align}
hold, which can be shown by simple calculations. 
We also note that if $G$ is a pure one-dimensional simplicial complex on $V$ (i.e., a graph without isolated vertices), then $L_0^{\text{up}} = \cL[G]$ from \eqref{eq:upLap}. 
By combining the fact that the transpose of $A[G]$ is a stochastic matrix, the eigenvalues of $\cL[G]$ are all real and lie between $0$ and $2$.
\subsection{Localization}
Let $X$ be a pure $D$-dimensional simplicial complex, and let $\tau = (v_0, \ldots, v_j)\in\Sg(X_j)$ be a fixed ordered $j$-simplex in $X$ with $-1\le j \le D$. We write $\lk_X(\tau)$ for $\lk_X(\{v_0, \ldots, v_j\})$ and define $m_\tau(\eta) = m(\tau\eta)$ for $\eta\in\lk_X(\tau)$. In other words, $m_\tau(\eta)$ is the number of $(D-j-1)$-simplices in $\lk_X(\tau)$ containing $\eta$.  For $0 \le l \le D-j-1$, $\Sg(\lk_X(\tau)_l)$ denotes the set of all ordered $l$-simplices of $\lk_X(\tau)$. Since $\lk_X(\tau)$ is a pure $(D-j-1)$-dimensional simplicial complex, we can define various concepts for $\lk_X(\tau)$, as we did in the previous subsection for $X$, by replacing $X$ with $\lk_X(\tau)$. We distinguish the concepts from those for $X$ by adding the superscript (or subscript) $\tau$ in the notation. For example, the coboundary operator on $C^l(\lk_X(\tau))$ is denoted $d_l^\tau$. Other symbols are indicated by $(\cdot,\cdot)_\tau, \|\cdot\|_\tau, \dl_l^\tau$, and so on. 
\begin{definition}
Let $-1 \le j < k \le D$ and take $\tau\in\Sg(X_j)$ and $\varphi\in C^k(X)$. The localization $\varphi_\tau\in C^{k-j-1}(\lk_X(\tau))$ of $\varphi$ with respect to $\tau$ is defined by
\[
\varphi_\tau(\eta) = \varphi(\tau\eta)
\quad\text{for }\eta\in\Sg(\lk_X(\tau)_{k-j-1}).
\]
\end{definition}
A straightforward calculation gives the following identities. 
\begin{lemma}[{\cite[Lemmas~1.10 and~1.12, and Corollary~1.13]{BS}}]\label{lem:local}
For $D\ge 1$ and $\varphi\in C^{D-1}(X)$, the following identities hold. 
\begin{align*}
D\|\varphi\|^2 &= \frac1{(D-1)!}\sum_{\tau\in\Sg(X_{D-2})}\|\varphi_\tau\|_\tau^2, \\
\|d_{D-1}\varphi\|^2-\|\varphi\|^2 &= \frac1{(D-1)!}\sum_{\tau\in\Sg(X_{D-2})}(\| d_0^\tau\varphi_\tau\|_\tau^2-\|\varphi_\tau\|_\tau^2). 
\end{align*}
In particular, we have
\begin{equation}
\| d_{D-1}\varphi\|^2 = \frac1{(D-1)!}\sum_{\tau\in\Sg(X_{D-2})}\left\{\| d_0^\tau\varphi_\tau\|_\tau^2-(1-D^{-1})\|\varphi_\tau\|_\tau^2\right\}. \label{eq: dphi}
\end{equation}
\end{lemma}
\subsection{Upper bounds of Betti numbers}
Let $D \ge 1$ and let $X$ be a pure $D$-dimensional simplicial complex. Take $\tau \in X_{D-2}$ or $\tau \in \Sg(X_{D-2})$. Let $N_X(\tau)$ denote the number of vertices in $\lk_X(\tau)$.
The eigenvalues of $\cL[\lk_X(\tau)]$, including repeated values, are denoted by $\lm_i^\tau$ and the corresponding eigenvectors by $\psi_i^\tau$ $ (i=1, \dots, N_X(\tau))$. Without loss of generality, we may assume that $\lm_1^\tau=0$, $\psi_1^\tau$ is a constant vector, and $\{\psi_i^\tau\}_{i=1}^{N_X(\tau)}$ is an orthonormal basis of $C^0(\lk_X(\tau))$. 
We consider the orthogonal decomposition of $C^0(\lk_X(\tau))$:
\[
C^0(\lk_X(\tau)) = A_1^\tau \oplus A_2^\tau \oplus A_3^\tau, 
\]
where
\begin{align*}
A_1^\tau: &= \spa_{\R}\{\psi_1^\tau\}, \\
A_2^\tau: &= \spa_{\R}\{\psi_i^\tau \mid i \neq 1 \text{ and } \lm_i^\tau \le 1- D^{-1}\}, \\
A_3^\tau: &= \spa_{\R}\{\psi_i^\tau \mid\lm_i^\tau > 1- D^{-1}\}. 
\end{align*}
For $j=1,2,3$, $\pi_j^\tau$ denotes the orthogonal projection of $C^0(\lk_X(\tau))$ onto $A_j^\tau$.
We obtain the following formula by direct calculation.
\begin{lemma}[{\cite[Lemma~1.11]{BS}}]\label{lem:1projNorm}
For $\varphi\in C^{D-1}(X)$ and $\tau\in\Sg(X_{D-2})$,  
\[
\|\pi_1^\tau\varphi_\tau\|_\tau^2 = \frac{m(\tau)}{2}\bigl((\dl_{D-1}\varphi)(\tau)\bigr)^2.
\]
\end{lemma}
The following lemma is important for proving the cohomology vanishing theorem and its quantitative generalization. 
\begin{lemma}\label{lem:normIneq}
Let $X$ be a pure $D$-dimensional simplicial complex with $D\ge1$. 
Then, for each $\tau\in X_{D-2}$, there exists some $i\in\{1,\ldots,N_X(\tau)\}$ such that $\lm_i^\tau\ge1$. Moreover, for $\varphi \in C^{D-1}(X)$, 
\begin{equation}\label{eq:2projNorm}
\|\varphi\|^2\le \frac{1}{\lm D}\biggl(\frac2{(D-1)!}\sum_{\tau\in\Sg(X_{D-2})}\|\pi_2^\tau\varphi_\tau\|_{\tau}^2+\|\dl_{D-1}\varphi\|^2 + \|d_{D-1}\varphi\|^2\biggr), 
\end{equation}
where $\lm := \min\{\lm_i^\tau-(1-D^{-1})\mid\tau\in X_{D-2}\text{ and }\lm_i^\tau>1-D^{-1}\}>0$. 
In particular, if $\varphi\in \ker L_{D-1}$, then
\begin{equation}\label{eq:2projNorm'}
\|\varphi\|^2\le \frac2{\lm\cdot D!}\sum_{\tau\in\Sg(X_{D-2})}\|\pi_2^\tau\varphi_\tau\|_{\tau}^2.
\end{equation}
\end{lemma}
\begin{proof}
Note that the zeroth up Laplacian on $\lk_X(\tau)$ is equal to $ \cL[\lk_X(\tau)]$. For simplicity, we use $\cL$ to denote $ \cL[\lk_X(\tau)]$ in the following. 
Since $X$ is pure, by the definition of $\cL$, 
\[
\sum_{i=1}^{N_X(\tau)}\lm_i^\tau=\tr(\cL)=N_X(\tau)
\]
for each $\tau\in X_{D-2}$. This implies that $\lm_i^\tau\ge1$ for some $i\in\{1,\ldots,N_X(\tau)\}$. 

For an arbitrary $\tau\in\Sg(X_{D-2})$ and $\varphi\in C^{D-1}(X)$, 
\begin{align*}
\|d_0^\tau\varphi_\tau\|_\tau^2 &= (\cL\varphi_\tau, \varphi_\tau)_\tau = \sum_{i=1}^3(\cL\pi_i^\tau\varphi_\tau, \pi_i^\tau\varphi_\tau)_\tau\\
&\ge (\cL\pi_3^\tau\varphi_\tau, \pi_3^\tau\varphi_\tau)_\tau\ge\kappa\|\pi_3^\tau\varphi_\tau\|_\tau^2, 
\end{align*}
where $\kappa:=(1-D^{-1})+\lm\le2$. 
In the above equation, the first inequality follows from the non-negative definiteness of $\cL$. The second equality follows by expressing $\pi_3^\tau\varphi_\tau$ as a linear combination of the eigenvectors $\{\psi_i^\tau\}_i$ generating $A_3^\tau$.
From \eqref{eq: dphi}, we have 
\begin{align*}
\|d_{D-1}\varphi\|^2
&= \frac1{(D-1)!}\sum_{\tau\in\Sg(X_{D-2})}\left(\|d_0^\tau\varphi_\tau\|_\tau^2 - (1-D^{-1})\|\varphi_\tau\|_\tau^2\right)\\
&\ge \frac1{(D-1)!}\sum_{\tau\in\Sg(X_{D-2})}\left(\kappa\|\pi_3^\tau\varphi_\tau\|_\tau^2 - (1-D^{-1})\|\varphi_\tau\|_\tau^2\right)\\
&= \frac1{(D-1)!}\sum_{\tau\in\Sg(X_{D-2})}\left(\lm\|\pi_3^\tau\varphi_\tau\|_\tau^2 - (1-D^{-1})(\|\pi_1^\tau\varphi_\tau\|_\tau^2+\|\pi_2^\tau\varphi_\tau\|_\tau^2)\right).
\end{align*}
Therefore,
\begin{equation}\label{eq:3projNorm}
\frac\lm{(D-1)!}\sum_{\tau\in\Sg(X_{D-2})}\|\pi_3^\tau\varphi_\tau\|_\tau^2
\le \|d_{D-1}\varphi\|^2 + \frac{1-D^{-1}}{(D-1)!}\sum_{\tau\in\Sg(X_{D-2})}(\|\pi_1^\tau\varphi_\tau\|_\tau^2 + \|\pi_2^\tau\varphi_\tau\|_\tau^2). 
\end{equation}
Then, 
\begin{align*}
\lm D\,\|\varphi\|^2
&=\frac1{(D-1)!}\sum_{\tau\in\Sg(X_{D-2})}\lm\|\varphi_\tau\|_{\tau}^2 \qquad\text{(from Lemma~\ref{lem:local})}\\
&= \frac1{(D-1)!}\sum_{\tau\in\Sg(X_{D-2})}\lm\sum_{i=1}^3\|\pi_i^\tau\varphi_\tau\|_{\tau}^2\\
&\le \frac\kappa{(D-1)!}\sum_{\tau\in\Sg(X_{D-2})}(\|\pi_1^\tau\varphi_\tau\|_{\tau}^2 + \|\pi_2^\tau\varphi_\tau\|_{\tau}^2) + \| d_{D-1}\varphi\|^2\qquad\text{(from \eqref{eq:3projNorm})}\nonumber\\
&= \frac\kappa{(D-1)!}\sum_{\tau\in\Sg(X_{D-2})}\frac{m(\tau)}{2}((\dl_{D-1}\varphi)(\tau))^2\\
&\qad+\frac\kappa{(D-1)!}\sum_{\tau\in\Sg(X_{D-2})}\|\pi_2^\tau\varphi_\tau\|_{\tau}^2+\| d_{D-1}\varphi\|^2\quad\text{(from Lemma~\ref{lem:1projNorm})}\\
&\le \|\dl_{D-1}\varphi\|^2 + \frac2{(D-1)!}\sum_{\tau\in\Sg(X_{D-2})}\|\pi_2^\tau\varphi_\tau\|_{\tau}^2+\| d_{D-1}\varphi\|^2. 
\end{align*}
Thus, \eqref{eq:2projNorm} holds.
If $\varphi\in\ker L_{D-1}$, then the second and third terms in the right-hand side of \eqref{eq:2projNorm} vanish because
\[
\varphi \in \ker L_{D-1} = \ker L_{D-1}^{\text{down}} \cap \ker L_{D-1}^{\text{up}} = \ker\delta_{D-1}\cap\ker d_{D-1}
\]
from \eqref{eq:kerdel}, \eqref{eq:kerd}, and~\eqref{eq:kerLap}.
Thus, \eqref{eq:2projNorm'} holds. 
\end{proof}
Theorem~\ref{thm: cohomology vanishing theorem} can now be proved by combining Lemma~\ref{lem:normIneq} with \eqref{eq:kerLap}, but we proceed further.
First, we prove a variant of Theorem~\ref{thm:geneCVT} with the extra assumption that $X$ is pure $D$-dimensional. 
\begin{theorem}	\label{thm:puregeneCVT}
Let $X$ be a pure $D$-dimensional simplicial complex with $D\ge1$. Then, 
\[
\b_{D-1}(X) \le \sum_{\tau \in X_{D-2}} \gm(\lk_X(\tau);1-D^{-1}). 
\] 
\end{theorem}
\begin{proof}
For each $\tau\in X_{D-2}$, select an arbitrary $\hat\tau\in\Sg(X_{D-2})$ from the ordered sequences of elements of $\tau$.
Define a linear map
\[
F\colon\ker L_{D-1}\to\bigoplus_{\tau \in X_{D-2}}A_2^\tau
\]
by $F(\varphi) = (\pi_2^\tau\varphi_{\hat{\tau}})_{\tau\in X_{D-2}}$. 
From Lemma~\ref{lem:normIneq}, $F$ is injective.
Comparing the dimensionalities, we have
\[
\b_{D-1}(X) = \dim(\ker L_{D-1}) \le \sum_{\tau \in X_{D-2}}\dim A_2^\tau = \sum_{\tau \in X_{D-2}} \gm(\lk_X(\tau);1-D^{-1}). \qedhere
\]
\end{proof}
We now remove the assumption of pure $D$-dimensionality in Theorem~\ref{thm:puregeneCVT}. 
Let $D \ge 1$ and let $X$ be a simplicial complex. We assume $\dim X \ge D-1$ and consider the $(D-1)$-th Betti number. Let $M_{D-1}(X)$ be the set of all maximal $(D-1)$-simplices in $X$, namely, the set of all $(D-1)$-simplices that are not contained in any $D$-simplex.  We define a new simplicial complex $\overline{X}^D$ by adding vertices and simplices to $X_D$ as follows. For each $\sg\in M_{D-1}(X)$, we construct a $D$-simplex $s_\sg := \sg\sqcup\{v_\sg\}$ with a new vertex $v_\sg$. We let $\overline{X}^D$ be the simplicial complex generated by $X_D \cup \{s_\sg\}_{\sg\in M_{D-1}(X)}$, that is, the smallest simplicial complex that includes $X_D \cup \{s_\sg\}_{\sg\in M_{D-1}(X)}$.
Clearly, $\overline{X}^D$ is a pure $D$-dimensional simplicial complex. In addition, the identity
\begin{equation}
\b_{D-1}(X) = \b_{D-1}\bigl(\overline{X}^D\bigr)	\label{eq:BettiHuhen}
\end{equation}
holds.
Indeed, let $X^{[D]}$ denote the simplicial complex $X \cap \overline{X}^D$. In other words, $X^{[D]}$ is generated by $X_D\cup X_{D-1}$.
Since $\ker d_{D-1}$ and $\im d_{D-2}$ do not change from replacing $X$ with $X^{[D]}$, we have $\b_{D-1}(X) = \b_{D-1}(X^{[D]})$.  
By construction, $\overline{X}^D$ and $X^{[D]}$ are homotopy equivalent, which implies that $\b_{D-1}\bigl(\overline{X}^D\bigr)=\b_{D-1}(X^{[D]})$. 
Thus, \eqref{eq:BettiHuhen} holds. 

For the proof of the following lemma, we note a few simple facts about some graphs $G$. 
Let $\a\in[0,2)$. 
If $G$ consists of only one vertex or one edge and its two vertices, then $\gm(G;\a)=0$. Also, adding an isolated vertex or an isolated edge to a non-empty graph $G$ increases $\gm(G;\a)$ by exactly one. 
\begin{lemma}\label{lem:gmhuhen}
Let $D$ and $X$ be as stated above. 
For $\a\in[0,2)$,
\[
\sum_{\tau \in X_{D-2}} \gm\bigl(\lk_X(\tau)^{(1)};\a\bigr) = \sum_{\tau \in (\overline{X}^D)_{D-2}} \gm_{D-1}(\lk_{\overline{X}^D}(\tau);\a). 
\]
\end{lemma}
\begin{proof}
If $\tau\in X_{D-2}\setminus (\overline{X}^D)_{D-2}$, then $\lk_X(\tau) = \emptyset$, which implies that $\gm\bigl(\lk_X(\tau)^{(1)};\a\bigr) = 0$. Suppose $\tau\in (\overline{X}^D)_{D-2}\setminus X_{D-2}$. Then there exists some $\sg \in M_{D-1}(X)$ such that $\tau \subset s_\sg$ and $\tau\nsubseteq\sigma$. Since $s_\sg$ is the only $D$-simplex of $\overline{X}^D$ that contains $\tau$, $\lk_{\overline{X}^D}(\tau)$ consists of only the  isolated edge $s_\sg\setminus\tau$. Thus, $\gm\bigl(\lk_{\overline{X}^D}(\tau);\a\bigr) = 0$. If $\tau\in X_{D-2} \cap (\overline{X}^D)_{D-2}$, then $\gm\bigl(\lk_X(\tau)^{(1)};\a\bigr) = \gm\bigl(\lk_{\overline{X}^D}(\tau);\a\bigr)$ because $\lk_{\overline{X}^D}(\tau)$ is obtained from $\lk_{X^{(D)}}(\tau)$ by replacing its isolated vertices with isolated edges. Therefore, 
\begin{align*}
\sum_{\tau \in X_{D-2}} \gm\bigl(\lk_X(\tau)^{(1)};\a\bigr) 
&= \sum_{\tau \in (X^{[D]})_{D-2}} \gm(\lk_{X^{[D]}}(\tau);\a)\\
& = \sum_{\tau \in (\overline{X}^D)_{D-2}} \gm\bigl(\lk_{\overline{X}^D}(\tau);\a\bigr). \myqedhere
\end{align*}
\end{proof}
We can now prove Theorem~\ref{thm:geneCVT}.
\begin{proof}[Proof of Theorem~\ref{thm:geneCVT}]
We may assume that $\dim X \ge D-1$, since otherwise both sides of \eqref{eq:geneCVT} vanish.
Noting that $\overline{X}^D$ is a pure $D$-dimensional simplicial complex, we have
\begin{align*}
\b_{D-1}(X) &= \b_{D-1}\bigl(\overline{X}^D\bigr) \qquad\text{(from \eqref{eq:BettiHuhen})}\\
&\le \sum_{\tau \in (\overline{X}^D)_{D-2}} \gm\bigl(\lk_{\overline{X}^D}(\tau);1-D^{-1}\bigr) \qquad\text{(from Theorem~\ref{thm:puregeneCVT})}\\
&= \sum_{\tau \in X_{D-2}} \gm\bigl(\lk_X(\tau)^{(1)};1-D^{-1}\bigr). \qquad\text{(from Lemma~\ref{lem:gmhuhen}) }\myqedhere
\end{align*}
\end{proof}
As a corollary, the assumption of pure $D$-dimensionality in Theorem~\ref{thm: cohomology vanishing theorem} can be removed. 
\begin{cor}	\label{cor:nonpureCVT}
Let $D\ge 1$, and let $X$ be a simplicial complex such that $\lm_2[\lk_X(\tau)] > 1 - D^{-1}$ for every $(D-2)$-simplex $\tau\in X$.
Then, $H^{D-1}(X) = \{0\}$. 
\end{cor}
\section{Estimates of Betti numbers of random simplicial complexes}
\subsection{Statement of results}
In this section, we consider multi-parameter random simplicial complexes, which were introduced  in \cite{CF,Fo}, and give some estimates of their Betti numbers. Linial--Meshulam complexes~\cite{LM1} and random clique complexes~\cite{K2} are shown as typical examples in this framework. 

Let $n\in\N$ and $\p = (p_0, p_1, \ldots, p_{n-1})$ be a multi-parameter with $0\le p_i \le 1$ for all $i = 0, 1, \ldots, n-1$. We start with the set $V$ of $n$ vertices and retain each vertex with independent probability $p_0$. Next, each edge with both ends retained is added with independent probability $p_1$. Iteratively, for $i=2,3,\dots,n-1$, each $i$-simplex for which all faces were added by this procedure is added to our complex with independent probability $p_i$. The distribution of the resulting random simplicial complexes $X$ is denoted by $X(n, \p)$. We call this model the multi-parameter random complex model with $n$ vertices and multi-parameter $\p$. 
\begin{example}\label{ex:LMcomplex}
Let $n>d \ge 1$ and let $0\le p\le 1$ be fixed. Define $\p = (p_0, p_1, \ldots, p_{n-1})$ by
\[
p_i := \begin{cases}
1		& (0\le i \le d-1), \\
p 		& (i = d), \\
0 		& (d+1 \le i \le n-1). 
\end{cases}
\]
The corresponding random simplicial complex follows the $d$-Linial--Meshulam complex model $Y_d(n,p)$. 
The Erd\H{o}s--R\'enyi graph model $G(n,p)$ is identified with $Y_1(n,p)$.
\end{example}
\begin{example}\label{ex:CLcomplex}
Let $n>d \ge 1$ and let $0\le p\le 1$ be fixed. Define $\p = (p_0, p_1, \ldots, p_{n-1})$ by
\[
p_i := \begin{cases}
1		& (0\le i \le d-1), \\
p 		& (i = d), \\
1 		& (d+1 \le i \le n-1). 
\end{cases}
\]
We call the corresponding random simplicial complex $C_n^{(d)}(p)$ the random $d$-flag complex. The random clique complex $C_n(p)$ is identical to the random $1$-flag complex $C_n^{(1)}(p)$. 
\end{example}
We state several estimates for the Betti numbers of the multi-parameter random complexes $X(n, \p)$ according to the dependence of $\p$ on $n$. Their proofs are left to subsequent subsections. To state the propositions, we introduce some notation:
\begin{equation}\label{eq:parameters}
\begin{gathered}
q_{-1} := 1,\quad q_k := \prod_{i=0}^{k} p_i^{\binom{k+1}{i+1}}\quad (0\le k\le n-1),\\
r_k := \frac{q_{k+1}}{q_k} = \prod_{i=0}^{k+1} p_i^{\binom{k+1}{i}} \quad(-1\le k\le n-2). 
\end{gathered}
\end{equation}
Here, we set $0/0=0$ and $\binom{0}{0}=1$ per convention.
Note that $0\le q_k \le 1$, $0\le r_k \le 1$, and $\P(\sg\in X) = q_k$ for any $\sg\in\binom{V}{k+1}$.
Moreover, both $q_k$ and $r_k$ are nonincreasing with respect to $k$. 

In what follows, $k\ge0$ is always fixed.
The following proposition follows from an easy application of the Morse inequality (see  also Section~6 of~\cite{Fo}).
\begin{prop}\label{prop:byMorse}
Let $c_1>0$ and $c_2\ge0$ satisfy $\frac{k+1}{c_1}+\frac{c_2}{k+2}<1$. Then, there exist some $n_0\in\N$ and $\eps_0>0$, depending only on $k$, $c_1$, and $c_2$, such that if $n\ge n_0$, then
\begin{equation}\label{eq:byMorse}
r_{k-1} \ge \frac{c_1}{n}\quad\text{and}\quad r_k \le \frac{c_2}{n}
\end{equation}
together imply
$\E[\b_k(X)]\ge \eps_0 n^{k+1}q_k$.
\end{prop}
The next result is an upper estimate for $r_k$ sufficiently large. This is a generalization of the cohomology vanishing theorem to multi-parameter random complexes (see, e.g., \cite[Theorem~1.1 (1)]{K1} and \cite[Theorem~2]{Fo}).
\begin{prop}\label{prop:CFvanish}
Let $\rho\ge1$ and $\dl>0$. Then, there exists a $K_0>0$, depending only on $k$, $\rho$, and $\dl$, such that if
\begin{equation}\label{eq:CFvanish}
r_{k-1}\ge\frac{K_0}{n}\quad\text{and}\quad r_k \ge \frac{(\rho+\dl)\log (n r_{k-1})}{n},
\end{equation}
then $\P(\b_k(X)\ne0) \le n^{k+1}q_k(n r_{k-1})^{-\rho}$ and $\E[\b_k(X)]\le n^{k+1}q_k(n r_{k-1})^{1-\rho}$. 
\end{prop}
\begin{cor}\label{cor:CFvanish}
Let $\theta\in(0,1]$, $\nu\ge\theta$, $\dl>0$, and $M>0$. Then, for sufficiently large $n$, 
\begin{equation}\label{eq:CFvanish2}
r_k \ge \frac{(\nu+\dl)\log n}{n}\quad\text{and}\quad M r_k^{1-\theta}\ge r_{k-1}
\end{equation}
imply
\[
\P(\b_k(X)\ne0) \le n^{k+1}q_k(n r_{k-1})^{-\nu/\theta}
\text{ and }
\E[\b_k(X)]\le n^{k+1}q_k(n r_{k-1})^{1-\nu/\theta}.
\] 
\end{cor}
The last result is a general upper estimate, which is the main result of this section. 
\begin{theorem}\label{thm:CFdecayBetti}
Let $l\in\N$. There exists a constant $C\ge0$ depending only on $k$ and $l$ such that for all $n\in\N$,
\[
\E[\b_k(X)] \le n^{k+1}q_k\{1\wedge C(n r_k)^{-l}\}.
\]
Here, if $r_k=0$, then the right-hand side is interpreted as $n^{k+1}q_k$.
\end{theorem}
We apply these results to typical examples.
\begin{example}[Linial--Meshulam complex]\label{ex:LM}
Consider the $d$-Linial--Meshulam complex $X\sim Y_d(n,p)$, as in Example~\ref{ex:LMcomplex}.
Letting $k=d-1$, we obtain $q_k=1$, $r_{k-1}=1$, and $r_k=p$.
Theorem~\ref{thm:CFdecayBetti}, Proposition~\ref{prop:byMorse}, and Corollary~\ref{cor:CFvanish} with $\theta=M=1$ together imply that
for given $0<c<k+2$, $l\in\N$, and $\nu'>\nu\ge1$, there exist $\eps_0>0$ and $C\ge0$ such that the following hold.
\begin{enumerate}
\item For every $n\in\N$,
\[
\E[\b_{d-1}(X)] \le n^d \{1\wedge C(np)^{-l}\}.
\] 
\item For sufficiently large $n$, if $p\le c/n$, then
\[\E[\b_{d-1}(X)]\ge \eps_0 n^d.\]
\item For sufficiently large $n$, if $p \ge {(\nu'\log n)}/n$, then
\[
\P(\b_{d-1}(X)\ne 0) \le n^{d-\nu}
\quad\text{and}\quad
\E[\b_{d-1}(X)]\le n^{d+1-\nu}. 
\]
\end{enumerate}
\end{example}
\begin{example}[Random clique complex]
Consider the random clique complex $X\sim C(n, p)$, as in Example~\ref{ex:CLcomplex}.
For $0\le k\le n-1$, we obtain $q_k=p^{\binom{k+1}2}$, $r_{k-1}=p^k$, and $r_k=p^{k+1}$. Here, $\binom{1}{2}=0$ per convention.
Theorem~\ref{thm:CFdecayBetti}, Proposition~\ref{prop:byMorse}, and Corollary~\ref{cor:CFvanish} with $\theta=1/(k+1)$ and $M=1$ together imply that, 
 for given $c_1>0$ and $c_2>0$ with $(k+1)/c_1+c_2/(k+2)<1$, $l\in\N$, and $\nu'>\nu\ge1/(k+1)$, there exist $\eps_0>0$ and $C\ge0$ such that the following hold.
\begin{enumerate}
\item For every $n\in\N$,
\[
\E[\b_k(X)] \le n^{k+1}p^{\binom{k+1}2}\{1\wedge C(np^{k+1})^{-l}\}.
\] 
\item For sufficiently large $n$, if $(c_1/n)^{1/k}\le p\le (c_2/n)^{1/(k+1)}$, then 
\[
  \E[\b_k(X)]\ge \eps_0 n^{k+1}p^{\binom{k+1}2}.
\]
\item For sufficiently large $n$, if $p \ge \{(\nu'\log n)/n\}^{1/(k+1)}$, then 
\begin{align*}
\P(\b_{k}(X) \ne 0) &\le n^{k+1}p^{\binom{k+1}2}(n p^k)^{-(k+1)\nu}\\
&=n^{k/2+1-\nu}(n p^{k+1})^{k/2-k\nu}\\
\shortintertext{and}
\E[\b_{k}(X)]&\le  n^{k+1}p^{\binom{k+1}2}(n p^k)^{1-(k+1)\nu}.
\end{align*}
\end{enumerate}
Note that when $\nu'> k/2+1+\eta$ for some $\eta\ge0$ in (3), the first conclusion implies $\P(\b_{k}(X) \ne 0)=o(n^{-\eta})$. 
When $\eta=0$, this claim is consistent with Theorem~1.1 (1) in \cite{K1}.
\end{example}
\subsection{Proofs of Propositions~\ref{prop:byMorse}~and \ref{prop:CFvanish}}
We follow \cite[Section~7]{K2} for the proof of Proposition~\ref{prop:byMorse}. 
For $k\ge0$, let $f_k(X)$ denote the cardinality of $X_k$, the set of all $k$-simplices of $X$. We set $f_{-1}(X) = 1$ per convention. 
The following inequality holds for arbitrary simplicial complexes.    
\begin{lemma}[{A version of the Morse inequality}]
\label{lem:MorseIneq}
Let $X$ be a simplicial complex. For every $k\ge0$, 
\begin{equation}\label{eq:MorseIneq}
f_k(X)-f_{k+1}(X)-f_{k-1}(X)\le\b_k(X)\le f_k(X). 
\end{equation}
\end{lemma}
\begin{proof}
Since $f_k(X) = \dim\ker d_k + \rank d_k$, we have
\begin{align*}
\b_k(X) &= \dim\ker d_k-\rank d_{k-1}\\
& = (f_k(X)-\rank d_k) - \rank d_{k-1}\\
& \ge f_k(X) - f_{k+1}(X) - f_{k-1}(X)
\end{align*}
and
\[
\b_k(X) \le \dim\ker d_k \le f_k(X).\qedhere
\]
\end{proof}
\begin{proof}[Proof of Proposition~\ref{prop:byMorse}]
Choose $n_0>k$ such that 
\[
\frac{n_0(k+1)}{(n_0-k)c_1}+\frac{c_2}{k+2}<1. 
\]
Then, \eqref{eq:byMorse} implies that, for $n\ge n_0$,
\begin{align*}
\E[f_{k+1}(X)] &= \binom{n}{k+2}q_{k+1} =  \frac{(n-k-1)r_k}{k+2}\E[f_{k}(X)]\le\frac{c_2}{k+2}\E[f_{k}(X)]\\
\shortintertext{and}
\E[f_{k-1}(X)] &= \binom{n}{k}q_{k-1} =  \frac{k+1}{(n-k)r_{k-1}}\E[f_{k}(X)]\le \frac{n_0}{n_0-k}\frac{k+1}{c_1}\E[f_{k}(X)].
\end{align*}
Combining these estimates with the first inequality of \eqref{eq:MorseIneq} yields the desired inequality.
\end{proof}
We now turn to proving Proposition~\ref{prop:CFvanish}.
The following theorem states that the spectral gap of the Erd\H{o}s--R\'{e}nyi graph $G \sim G(n,p)$ concentrates around $1$ if the parameter $p$ is sufficiently large. 
\begin{theorem}[{Theorem 1.1 in \cite{HKP}, spectral gap theorem in \cite{K1}}]\label{thm:ERspegap}
Let $G\sim G(n,p)$ be the Erd\H{o}s--R\'{e}nyi graph.
Let $\eta\ge0$, $\dl>0$, and $\eps>0$.
Then, for sufficiently large $n$, if
\[
p \ge \frac{(1+\eta+\dl)\log n}n,
\]
then $\P(\lm_2[G] > 1- \eps) \ge 1 - \eps n^{-\eta}$. 
\end{theorem}
The following lemma concerns the structure of links in multi-parameter random complexes. 
Let $X\sim X(n, \p)$ be a multi-parameter random complex with an $n$-point vertex set $V$ that is defined on a probability space $(\Om,\cF,\P)$. 
For a simplex $\tau$ with $\P(\tau\in X)>0$, we define a probability space $(\Om_\tau, \cF_\tau, \P_\tau)$ by
\[
\Om_\tau := \{\tau\in X\}, \,\cF_\tau := \{B\in\cF\mid B\subset \Om_\tau \}, \text{ and } \P_\tau(\cdot) := \P(\cdot\mid\Om_\tau ). 
\]
Let $V_X(\tau)$ denote the vertex set of $\lk_X(\tau)$ and $N_X(\tau)$ denote its cardinality; these are random variables on $\Om_\tau$. 
The expectation with respect to $\P_\tau$ is denoted by $\E_\tau$. 
Let $\bin(n,p)$ indicate the binomial distribution with parameters $n$ and $p$. 
\begin{lemma}\label{lem:CFlink}
Let $0\le k \le n-2$ and consider $\tau\in\binom{V}{k}$. Provided that $\P(\tau\in X)>0$, the distribution of $(\lk_X(\tau))^{(1)}$ under $\P_\tau$ is $X(n-k,(r_{k-1},r_k/r_{k-1},0,\dots,0))$. 
In particular, the distribution of $N_X(\tau)$ under $\P_\tau$ is $\bin(n-k, r_{k-1})$. 
\end{lemma}
\begin{proof}
A vertex $v\in V\setminus\tau$ belongs to $V_X(\tau)$ if and only if $X$ contains every possible simplex that can be described as the union of $\{v\}$ and a subset of $\tau$. For $0\le i\le k$, there are $\binom{k}{i}$ such $i$-simplices. This implies
\[
\P_\tau(v\in V_X(\tau)) = \prod_{i=0}^k p_i^{\binom{k}{i}} = r_{k-1}. 
\]
Moreover, events $\{v\in V_X(\tau)\}_{v\in V\setminus\tau}$ are independent under $\P_\tau$ since distinct events are described in terms of distinct simplices. 
Let $\hat V\subset V\setminus\tau$ with $\P_\tau(V_X(\tau) = \hat V)>0$. 
An edge between vertices $v$ and $w$ in $\hat V$  belongs to $\lk_X(\tau)$ if and only if $X$ contains every possible simplex described as the union of $\{v,w\}$ and a subset of $\tau$. For $1\le i\le k+1$, there are $\binom{k}{i-1}$ such possible $i$-simplices. This implies
\[
\P_\tau(\{v,w\}\in \lk_X(\tau)\mid V_X(\tau) = \hat V) = \prod_{i=1}^{k+1} p_i^{\binom{k}{i-1}} = \frac{r_k}{r_{k-1}}. 
\]
Moreover, events $\{e\in \lk_X(\tau)\}_{e\in\binom{\hat V}{2}}$ are independent under $\P_\tau(\cdot\mid V_X(\tau) = \hat V)$ by the same reasoning as used above. Thus, the claim holds. 
\end{proof}
\begin{proof}[Proof of Proposition~\ref{prop:CFvanish}]
From Theorem~\ref{thm:geneCVT},
\begin{align*}
\E[\b_k(X)]
&\le \sum_{\tau\in\binom{V}{k}}\E[N_X(\tau); \tau\in X_{k-1},\ \lm_2[\lk_X(\tau)^{(1)}] \le 1-(k+1)^{-1}]\\
&= \sum_{\tau\in\binom{V}{k}}\P(\tau\in X_{k-1})\E_{\tau}[N_X(\tau); \lm_2[\lk_X(\tau)^{(1)}] \le 1-(k+1)^{-1}]\\
&= \binom{n}{k}q_{k-1}\,\E[N; \lm_2[Z] \le 1-(k+1)^{-1}], 
\end{align*}
where $Z\sim X(n-k,(r_{k-1},r_k/r_{k-1},0,\dots,0))$ and $N$ is the number of vertices of $Z$, which follows $\bin(n-k, r_{k-1})$.
The last identity follows from Lemma~\ref{lem:CFlink}.   Define $\mu := \E[N] = (n-k)r_{k-1}$ and
recall the Chernoff bound 
\begin{equation}\label{eq:chernoff}
\P[|N - \mu| > \mu^{3/5}]\le \exp(-\mu^{1/5}/5).
\end{equation}
If $K_0\ge 2k$ and \eqref{eq:CFvanish} holds, then $\mu\ge n r_{k-1}-k \ge n r_{k-1}/2\ge K_0/2$. 
Thus,
\[
\frac{r_k}{r_{k-1}}
\ge \frac{(\rho+\dl)\log(n r_{k-1})}{n r_{k-1}}
\ge\sup_{m\ge \lfloor\mu-\mu^{3/5}\rfloor}\frac{(\rho+\dl/2)\log m}{m}
\]
for $K_0$ sufficiently large. 
By combining these estimates with Theorem~\ref{thm:ERspegap}, for $\eps=(k+1)^{-1}\wedge  2^{-\rho}$, we have
\begin{align*}
&\E[N; \lm_2[Z] \le 1-(k+1)^{-1}]\\
&\le\sum_{m=\lfloor\mu-\mu^{3/5}\rfloor}^{\lceil\mu+\mu^{3/5}\rceil} m \P(\lm_2[Z] \le 1-(k+1)^{-1} \mid N=m)\P(N=m)
+\E[N; |N-\mu|>\mu^{3/5}]\\
&\le \sum_{m=\lfloor\mu-\mu^{3/5}\rfloor}^{\lceil\mu+\mu^{3/5}\rceil}m \eps m^{1-\rho}\P(N=m) +\E[N^2]^{1/2}\P(|N-\mu|>\mu^{3/5})^{1/2}\\
&\le \eps \mu(\lfloor\mu-\mu^{3/5}\rfloor)^{1-\rho}+(\mu+\mu^2)^{1/2}\exp(-\mu^{1/5}/10)\\
&\le 2\eps\mu^{2-\rho}
\end{align*}
when $K_0$ is sufficiently large.
Here, note that the distribution of $\lm_2[Z]$ under $\P(\cdot\mid N=m)$ is that of $\lm_2[G]$ with $G\sim G(m,r_k/r_{k-1})$.
Therefore, 
\begin{align*}
\E[\b_k(X)]
&\le 2\eps\binom{n}{k}q_{k-1}\mu^{2-\rho} \\
&\le 2\eps n^k q_{k-1}n r_{k-1}(n r_{k-1}/2)^{1-\rho}\\
&\le n^{k+1} q_k (n r_{k-1})^{1-\rho}. 
\end{align*}
The estimate of $\P(\b_k(X) \ne 0)$ is obtained in the same way. In this case, from Corollary~\ref{cor:nonpureCVT}, we have
\begin{align*}
\P(\b_k(X) \ne 0)
&\le \P\Bigl(\text{There exists }\tau\in X_{k-1}\text{ such that }\lm_2[\lk_X(\tau)^{(1)}] \le 1-(k+1)^{-1}\Bigl)\\
&\le \sum_{\tau\in\binom{V}{k}}\P(\tau\in X_{k-1},\ \lm_2[\lk_X(\tau)^{(1)}] \le 1-(k+1)^{-1})\\
&= \binom{n}{k}q_{k-1}\,\P(\lm_2[Z] \le 1-(k+1)^{-1}). 
\end{align*}
Then, for $\eps=(k+1)^{-1}\wedge 2^{-\rho}$,
\begin{align*}
&\P(\lm_2[Z] \le 1-(k+1)^{-1})\\
&\le\sum_{m=\lfloor\mu-\mu^{3/5}\rfloor}^{\lceil\mu+\mu^{3/5}\rceil} \P(\lm_2[Z] \le 1-(k+1)^{-1} \mid N=m)\P(N=m)+\exp(-\mu^{1/5}/5)\\
&\le\eps(\lfloor\mu-\mu^{3/5}\rfloor)^{1-\rho}+\exp(-\mu^{1/5}/5)
\le 2\eps\mu^{1-\rho}
\end{align*}
for $K_0$ sufficiently large.
Therefore, 
\begin{align*}
\P(\b_k(X) \ne 0)
&\le 2\eps\binom{n}{k}q_{k-1}\mu^{1-\rho} \\
&\le 2\eps n^k q_{k-1}(n r_{k-1}/2)^{1-\rho}\\
&\le n^{k+1} q_k (n r_{k-1})^{-\rho}. \myqedhere
\end{align*}
\end{proof}
\begin{proof}[Proof of Corollary~\ref{cor:CFvanish}]
Take $K_0$ in Proposition~\ref{prop:CFvanish} in which we let $\rho$ and $\dl$ be $\nu/\theta$ and $\dl/(2\theta)$, respectively. From \eqref{eq:CFvanish2}, for sufficiently large $n$, we have $r_{k-1}\ge r_k\ge K_0/n$ and
\[
r_k \ge \frac{(\nu/\theta+\dl/(2\theta))\log (n r_{k-1})}{n}. 
\]
Indeed, if we set $\bar r = (\nu+\dl)(\log n)/n$, then, for sufficiently large $n$, 
\begin{align*}
&n r_k-(\nu/\theta+\dl/(2\theta))\log{(n r_{k-1})}\\
&\ge n r_k-(\nu/\theta+\dl/(2\theta))\log{\bigl(n M r_k^{1-\theta}\bigr)}\\
&= n r_k-(\nu/\theta+\dl/(2\theta))(1-\theta)\log{(n r_k)} - (\nu+\dl/2)(\log{n}-(\log M)/\theta)\\
&\ge n \bar r-(\nu/\theta+\dl/(2\theta))(1-\theta)\log{(n \bar r)} - (\nu+\dl/2)(\log{n}-(\log M)/\theta)\\
&=\dl/2 \log n -(\nu/\theta+\dl/(2\theta))\{(1-\theta)(\log\log n+\log(\nu+\dl))+\log M\}\\
&\ge 0. 
\end{align*}
The conclusion follows from Proposition~\ref{prop:CFvanish}. 
\end{proof}
\subsection{Proof of Theorem~\ref{thm:CFdecayBetti}}
Theorem~\ref{thm:geneCVT} plays a key role in proving Theorem~\ref{thm:CFdecayBetti}.
We first discuss the eigenvalues of the averaging operator on the Erd\H{o}s--R\'{e}nyi graph.

Let $G=V\sqcup E$ be a graph and let $h\in\N$. We call $w=(v_0, v_1, \ldots, v_h)\in V^{h+1}$ a walk on $G$ with length $h$ if $v_i$ is adjacent to $v_{i+1}$ for all $i=0, 1, \ldots, h-1$. A walk $w=(v_0, v_1, \ldots, v_h)\in V^{h+1}$ where $v_0=v_h$ is called a closed walk. We denote by $W_h(G)$ the set of all length-$h$ closed walks on $G$. 
Given a graph $G$ and a closed walk $w=(v_0, v_1, \ldots, v_h)\in W_h(G)$, let $G(w) = V(w) \sqcup E(w)$ denote the subgraph of $G$ induced by $w$, where
\[
V(w) := \{v_0, v_1, \ldots, v_{h-1}\} \text{ and }  E(w) := \{\{v_0, v_1\}, \{v_1, v_2\}, \ldots, \{v_{h-1}, v_h\}\}
\]
are the vertex set and the edge set, respectively. The multiplicity $m_s(w)$ of $s\in V(w)$ is defined by
\[
m_s(w) := \# \{j\in\{0, 1, \ldots, h-1\}\mid v_j = s \}. 
\]
For $1\le v, e\le h$, the set $W_h^{v,e}(G)$ is the set of all $w\in W_h(G)$ such that $\#V(w) = v$ and $\#E(w) = e$, and $w_h^{v,e}$ is the number of length-$h$ closed walks on $v$ unlabeled vertices that traverse exactly $e$ edges (perhaps multiple times). 
That is, 
\[
w_h^{v,e} := \#W_h^{v,e}(K_v)\big/v!, 
\]
where $K_v$ is the complete graph with $v$ vertices.
\begin{lemma}	\label{lem: properties of closed walk}
The following properties hold for $l\in\N$ and $1 \le v, e \le 2l$. 
\begin{enumerate}
\item If $w_{2l}^{v,e} > 0$, then $e \ge v-1$.
\item If $w_{2l}^{e+1,e} > 0$, then $e \le l$.
\end{enumerate}
\end{lemma}
\begin{proof}
Take $w \in W_{2l}^{v,e}(K_v)$. By applying the Euler--Poincar\'{e} formula to the graph $G(w)$, we have $v-e = 1 -\b_1(G(w)) \le 1$. This implies (1). For the proof of (2), take $w \in W_{2l}^{e+1,e}(K_{e+1})$. Since $\b_1(G(w)) = 0$ (and $G(w)$ is connected), $G(w)$ is a tree. Then, since $w$ is a closed walk, $w$ passes through each edge of $G(w)$ at least twice. This implies $e \le l$. 
\end{proof}
\begin{lemma}\label{lem:ERspectra}
Let $G \sim G(n,p)$ be the Erd\H{o}s--R\'{e}nyi graph and let $\a > 0$. Let $\{\mu_i\}_{i=1}^n$ be all the (not necessarily distinct) eigenvalues of the averaging matrix $A[G]$. Then, for $l\in\N$ and $n \ge 2l$, 
\[
\E[\#\{i\mid\mu_i \ge \a\}] \le \frac{(2l)!}{\a^{2l}(n-2l+1)^{2l}p^{2l}}\sum_{\substack{1 \le v, e \le 2l;\\e \ge v-1}}w_{2l}^{v,e} n^v p^e + \frac{n(1-p)^{n-1}}{\a^{2l}}. 
\]
\end{lemma} 
\begin{proof}
To proceed, we identify the vertex set of $G$ with $\{1,2,\dots,n\}$. 
Let $\{X_{ij}\}_{1\le i < j \le n}$ be independent and identically distributed random variables that follow the Bernoulli distribution with parameter $p$. The Erd\H{o}s--R\'{e}nyi graph $G$ can be generated by $\{X_{ij}\}_{1\le i < j \le n}$: edge $\{i, j\}$ is supposed to belong to $G$ if and only if $X_{ij} = 1$. In addition, we define $X_{ii}=0$ for $1\le i\le n$ and $X_{ji} = X_{ij}$ for $1 \le i < j \le n$. 
Then, $a_{ij}$, the $(i, j)$-component of $A[G]$, is given by
\begin{align*}
a_{ij} &= \begin{cases}
X_{ij} \big/ \left(\sum_{s=1}^n X_{is}\right)		& \text{if $\sum_{s=1}^n X_{is}\neq 0$}, \\
1		& \text{if $\sum_{s=1}^n X_{is}=0$ and $i=j$}, \\
0		& \text{otherwise}. 
\end{cases}
\end{align*}
The obvious bound gives 
\[
\#\{ i\mid\mu_i \ge \a\} \le \frac{1}{\a^{2l}}\sum_{i=1}^n\mu_i^{2l} 
= \frac{1}{\a^{2l}}\tr(A[G]^{2l}). 
\]
Let $I(G)$ denote the number of isolated vertices of $G$. With this, 
\begin{align}\label{eq:trIndep}
\tr(A[G]^{2l}) &= \sum_{1 \le i_0, i_1, \ldots, i_{2l-1} \le n} a_{i_0i_1}a_{i_1i_2}\cdots a_{i_{2l-1}i_0}\\
&= \sum_{\substack{1 \le i_0, i_1, \ldots, i_{2l-1} \le n;\\i_0\neq i_1, i_1\neq i_2,\ldots, i_{2l-1}\neq i_0}} a_{i_0i_1}a_{i_1i_2}\cdots a_{i_{2l-1}i_0} + I(G)	\nonumber\\
&= \sum_{w=(i_0, i_1, \ldots, i_{2l}) \in W_{2l}(G)} a_{i_0i_1}a_{i_1i_2}\cdots a_{i_{2l-1}i_{2l}} + I(G)\nonumber\\
&= \sum_{w=(i_0, i_1, \ldots, i_{2l}) \in W_{2l}(G)}\prod_{j=0}^{2l-1}\frac{1}{\sum_{s_j = 1}^{n}X_{i_j s_j}} + I(G)
\nonumber\\
&\le \sum_{w=(i_0, i_1, \ldots, i_{2l}) \in W_{2l}(G)}\prod_{j=0}^{2l-1}\frac{1}{\sum_{s_j \notin V(w)}X_{i_j s_j}+1} + I(G). \nonumber
\end{align}
Here, in the second line, we used the fact that, for each $i=1,\dots,n$, $a_{ii} \neq 0$ if and only if $a_{is} = 0$ for every $s\neq i$. In the third line, recall that if vertex $i$ is not adjacent to vertex $j \neq i$, then $a_{ij} = 0$. 
The last inequality follows from the fact that each $i_j$ has at least one adjacent vertex in $V(w)$. 

By using the independence of $\{X_{ij}\}_{1\le i < j \le n}$, the expectation of the first term of the last line of \eqref{eq:trIndep} is equal to
\begin{align}\label{eq:inversBinom}
&\sum_{w=(i_0, i_1, \ldots, i_{2l}) \in W_{2l}(K_n)}\E\left[\prod_{j=0}^{2l-1}\frac{1}{\sum_{s_j \notin V(w)}X_{i_j s_j}+1}\,;\,w \in W_{2l}(G)\right]
\\
&= \sum_{w=(i_0, i_1, \ldots, i_{2l}) \in W_{2l}(K_n)}\P(w \in W_{2l}(G))\E\left[\prod_{j=0}^{2l-1}\frac{1}{\sum_{s_j \notin V(w)}X_{i_j s_j}+1}\right]
\nonumber\\
&= \sum_{\substack{1 \le v, e \le 2l;\\e \ge v-1}}\sum_{w=(i_0, i_1, \ldots, i_{2l}) \in W_{2l}^{v,e}(K_n)} p^e\prod_{i \in V(w) } \E\left[(Z_v+1)^{-m_i(w)}\right], \nonumber
\end{align}
where $Z_v \sim \bin(n-v, p)$. Here, Lemma~\ref{lem: properties of closed walk}(1) was used for the last identity. 
We also have, denoting $m_i(w)$ by $m_i$, 
\begin{align*}
&\E\left[ (Z_v+1)^{-m_i}\right]\\
&\le m_i!\,\E\left[ (Z_v+m_i)^{-1}(Z_v+m_i-1)^{-1}\cdots(Z_v+1)^{-1}\right]\\
&= m_i!\,\sum_{r=0}^{n-v} \frac{1}{(r+m_i)(r+m_i-1)\cdots(r+1)}\cdot\frac{(n-v)!}{r!\,(n-v-r)!}p^r(1-p)^{n-v-r}\\
&= \frac{m_i!}{(n-v+m_i)(n-v+m_i-1)\cdots(n-v+1)\,p^{m_i}}\\
&\qad\times \sum_{r=0}^{n-v} \frac{(n-v+m_i)!}{(r+m_i)!\,(n-v-r)!}p^{r+m_i}(1-p)^{n-v-r}\\
&\le \frac{m_i!}{(n-v+1)^{m_i}p^{m_i}}.
\end{align*}
Since $\sum_{i\in V(w)} m_i(w) = 2l$, we have $\prod_{i\in V(w)}\{m_i(w)!\}\le (2l)!$.
By combining these estimates, \eqref{eq:inversBinom} is dominated by
\begin{align*}
\sum_{\substack{1 \le v, e \le 2l;\\ e \ge v-1}}\sum_{w=(i_0, i_1, \dots, i_{2l})\in W_{2l}^{v,e}(K_n)} \frac{p^e\,(2l)!}{(n-v+1)^{2l}p^{2l}}
&\le \sum_{\substack{1 \le v, e \le 2l;\\e \ge v-1}}\frac{w_{2l}^{v,e} n^v p^e \,(2l)!}{(n-v+1)^{2l}p^{2l}}\\
&\le \frac{(2l)!}{(n-2l+1)^{2l}p^{2l}}\sum_{\substack{1 \le v, e \le 2l;\\e \ge v-1}}w_{2l}^{v,e} n^v p^e. 
\end{align*}
Since $\E[I(G)] = n(1-p)^{n-1}$, we reach the desired conclusion. 
\end{proof}
We remark that $\gm(G;1-\a)$ is $\#\{i\mid \mu_i\ge \a\}-1$ since $\cL[G]=I_V-A[G]$.
Now, we prove Theorem~\ref{thm:CFdecayBetti}.
\begin{proof}[Proof of Theorem~\ref{thm:CFdecayBetti}]
Since $\E[f_{k}(X)]=\binom{n}{k+1}q_k$, the second inequality of \eqref{eq:MorseIneq} implies that 
\begin{equation}\label{eq:trivial}
\E[\b_k(X)] \le n^{k+1}q_k. 
\end{equation}
Thus, it suffices to prove
\[
\E[\b_k(X)] \le C n^{k+1}q_k(n r_k)^{-l}
\]
for some constant $C$ that depends on only $k$ and $l$.
Take $K_0$ in Proposition~\ref{prop:CFvanish} with $\rho=l+1$ and $\dl=1$.
Take $K_1\ge K_0\vee1$ such that $x^{1/l}\ge (l+2)\log x$ for all $x\ge K_1$.
Suppose that 
\begin{equation}\label{eq:CFdecayBetti_first}
r_k\ge \frac{K_1}n\vee\frac{(n r_{k-1})^{1/l}}{n}.
\end{equation}
Then, $r_{k-1}\ge r_k\ge K_1/n$ and $r_k\ge(l+2)\{\log (n r_{k-1})\}/n$ hold. 
Thus, from Proposition~\ref{prop:CFvanish} with $\rho=l+1$ and $\dl=1$, 
\[
\E[\b_k(X)]\le n^{k+1}q_k (nr_{k-1})^{-l}\le n^{k+1}q_k (nr_k)^{-l}.
\]
Next, consider the constraint
\begin{equation}\label{eq:CFdecayBetti_middle}
\frac{K_2}{n}\le r_k \le \frac{(n r_{k-1})^{1/l}}{n}
\end{equation}
for some constant $K_2\ge K_1$ that will be specified later.
By applying Theorem~\ref{thm:geneCVT} to $X$ with $D=k+1$, we have
\begin{align*}
\E[\b_k(X)] &\le \E\left[\sum_{\tau \in X_{k-1}}\gm\bigl(\lk_X(\tau)^{(1)};1-(k+1)^{-1}\bigr)\right]\\
&= \sum_{\tau \in \binom{V}{k}}\P(\tau \in X_{k-1})\E_\tau\bigl[\gm\bigl(\lk_X(\tau)^{(1)};1-(k+1)^{-1}\bigr)\bigl]\\
&=\binom{n}{k} q_{k-1}\E[\gm\bigl(Z;1-(k+1)^{-1}\bigr)], 
\end{align*}
where $Z\sim X(n-k,(r_{k-1},r_k/r_{k-1},0,\dots,0))$.
The last identity follows from Lemma~\ref{lem:CFlink}. Denote by $N$ the number of vertices of $Z$, which follows $\bin(n-k, r_{k-1})$.  Define $\mu := \E[N] = (n-k)r_{k-1}$.
Take $K_2\ge K_1\vee 2k$ so that $\lfloor x-x^{3/5}\rfloor\ge 2l$ for all $x\ge K_2/2$ and suppose that \eqref{eq:CFdecayBetti_middle} holds. Consequently, we have $\mu \ge n r_{k-1}/2\ge K_2/2$ and $\lfloor\mu-\mu^{3/5}\rfloor\ge 2l$. Then, Lemma~\ref{lem:ERspectra} implies that
\begin{align}\label{eq:CFdecayBetti2}
&\E[\gm\bigl(Z;1-(k+1)^{-1}\bigr)]\\
&\le\sum_{m=\lfloor\mu-\mu^{3/5}\rfloor}^{\lceil\mu+\mu^{3/5}\rceil} \E[\gm\bigl(Z;1-(k+1)^{-1}\bigr)+1 \mid N=m]\,\P(N=m)\nonumber\\
&\qad+\E[\gm\bigl(Z;1-(k+1)^{-1}\bigr)+1;|N- \mu|> \mu^{3/5}]-1\nonumber\\
&\le\sum_{m=\lfloor\mu-\mu^{3/5}\rfloor}^{\lceil\mu+\mu^{3/5}\rceil} \Biggl\{\frac{(2l)! (k+1)^{2l}}{(m-2l+1)^{2l}(r_k/r_{k-1})^{2l}}\sum_{\substack{1 \le v, e \le 2l;\\e \ge v-1}}w_{2l}^{v,e}m^v \left(\frac{r_k}{r_{k-1}}\right)^e\nonumber\\
&\qad+(k+1)^{2l}m\left(1-\frac{r_k}{r_{k-1}}\right)^{m-1}\Biggr\}\P(N=m)
+\E[N;|N- \mu|> \mu^{3/5}]-1\nonumber\\
&\le (2l)!\,(k+1)^{2l}n r_{k-1}\sum_{\substack{1 \le v, e \le 2l;\\e \ge v-1}}\frac{w_{2l}^{v,e}A_v}{(nr_{k-1})^{2l-v+1}(r_k/r_{k-1})^{2l-e}}\nonumber\\
&\qad + (k+1)^{2l}\E\left[N\left(1-\frac{r_k}{r_{k-1}}\right)^{N-1}\right] + \E[N^2]^{1/2}\P(|N- \mu|> \mu^{3/5})^{1/2}-1,\nonumber
\end{align}
where 
\[
A_v=\sum_{m=\lfloor\mu-\mu^{3/5}\rfloor}^{\lceil\mu+\mu^{3/5}\rceil}\left(\frac{n r_{k-1}}{m-2l+1}\right)^{2l}\left(\frac{m}{n r_{k-1}}\right)^v \P(N=m).
\]
If $K_2$ (which depends on only $k$ and $l$) is chosen to be larger in advance, each $A_v$ becomes less than $2$.
In what follows, $C$ is a positive constant depending on only $k$ and $l$; it may vary from line to line.
Concerning the first term of the last line of \eqref{eq:CFdecayBetti2}, we have
\begin{align*}
&\sum_{\substack{1 \le v, e \le 2l;\\e \ge v-1}}\frac{w_{2l}^{v,e}}{(nr_{k-1})^{2l-v+1}(r_k/r_{k-1})^{2l-e}}\\
&= \sum_{m=0}^{2l}\sum_{\substack{1 \le v, e \le 2l;\\v-e=1-m}}\frac{w_{2l}^{v,e}}{(n r_{k-1})^m(n r_k)^{2l-e}}\\
&= \sum_{e=1}^{l}\frac{w_{2l}^{e+1,e}}{(n r_k)^{2l-e}} + \sum_{m=1}^{2l}\sum_{\substack{1 \le v, e \le 2l;\\v-e=1-m}}\frac{w_{2l}^{v,e}}{(n r_{k-1})^m(n r_k)^{2l-e}}\quad\text{(from Lemma~\ref{lem: properties of closed walk}(2))}\\
&\le (n r_k)^{-l}\sum_{e=1}^{l}w_{2l}^{e+1,e} + (n r_{k-1})^{-1}\sum_{m=1}^{2l}\sum_{\substack{1 \le v, e \le 2l;\\v-e=1-m}}w_{2l}^{v,e}\\
&\le C(n r_k)^{-l}. 
\end{align*}
Here, the first inequality follows from the relations $nr_{k-1}\ge n r_k\ge K_2\ge1$, and the last one follows from the inequality $(nr_{k-1})^{-1} \le (n r_k)^{-l}$ in \eqref{eq:CFdecayBetti_middle}. 

Considering the second term of \eqref{eq:CFdecayBetti2},
\begin{align*}
\E\bigl[N(1-r_k/r_{k-1})^{N-1}\bigr] &= \E[I(Z)]\\
&= \sum_{v\in \{1,2,\dots,n-k\}}\P\bigl(\text{$v$ is an isolated vertex in $Z$}\bigr)\\
&=(n-k) r_{k-1} (1-r_k)^{n-k-1}\\
&\le C n r_{k-1} (n r_k)^{-l}.
\end{align*}
The last inequality follows from the inequalities $(1-x)^m\le e^{-m x}\le C(m x)^{-l}$ for $0\le x\le1$ and $m\ge0$, and $\exp(r_k(k+1))\le e^{k+1}$.

The third term of \eqref{eq:CFdecayBetti2} is dominated by
$(\mu^2+\mu)^{1/2}\exp(-\mu^{1/5}/10)$
from the Chernoff bound \eqref{eq:chernoff}, which is less than $1$ for $K_2$ greater than some absolute constant.

Combining these estimates, we obtain 
\[
\E[\b_k(X)]\le C n^k q_{k-1}n r_{k-1}(n r_k)^{-l} = C n^{k+1} q_k (n r_k)^{-l}.
\]

Lastly, if
\begin{equation}\label{eq:CFdecayBetti_last}
r_k \le \frac{K_2}{n},
\end{equation}
then, from \eqref{eq:trivial},
\[
\E[\b_k(X)]\le n^{k+1}q_k\le K_2^{l}n^{k+1}q_k(n r_k)^{-l}.
\]

Since \eqref{eq:CFdecayBetti_first}, \eqref{eq:CFdecayBetti_middle}, and \eqref{eq:CFdecayBetti_last} cover all cases, the proof is completed.
\end{proof}
\section{Estimates of lifetime sums of random simplicial complex processes}
\subsection{Persistent homology}
Let $X$ be a simplicial complex.  A family $\cX = \{X(t)\}_{t\ge 0}$ of subcomplexes of $X$ is a right-continuous filtration of $X$ if $X(s)\subset X(t)$ for $0\le s \le t$ and $X(t) = \bigcap_{t' > t}X(t')$ for $t\ge 0$. Since $X$ is a finite simplicial complex, $X(t)$ differs from $\bigcup_{t' < t}X(t')$ only finitely many times. 
Here, $X(t)$ can be the empty set, which is considered to be a $(-1)$-dimensional simplicial complex.

Let $\R[\R_{\ge 0}]$ be an $\R$-vector space of formal linear combinations of finite elements in $\R_{\ge 0}$. 
Each element of $\R[\R_{\ge 0}]$ is expressed as a linear combination of monomials $z^t$ ($t\in\R_{\ge 0}$), where $z$ is an indeterminate. The product of two elements is given by the linear extension of $az^s\cdot bz^t := abz^{s+t}$~($a, b\in\R$, $s, t\in\R_{\ge 0}$). This operation equips $\R[\R_{\ge 0}]$ with a graded ring structure. 
For $k\ge0$, the $k$-th persistent homology $\ph_k(\cX)$ of $\cX = \{X(t)\}_{t\ge 0}$ is defined as the result of taking a direct sum in the $k$-th homology vector space: 
\[
\ph_k(\cX) := \bigoplus_{t\ge 0}H_k(X(t)). 
\]
It follows that $\ph_k(\cX)$ has a graded module structure over the graded ring $\R[\R_{\ge 0}]$ with isomorphisms from $H_k(X(s))$ to $H_k(X(t))$ ($0\le s\le t$) induced by the inclusion from $X(s)$ to $X(t)$. The following theorem is called the structure theorem of the persistent homology. 
\begin{theorem}[{\cite{ZC,HS}}]
For each $k\ge 0$, there exist unique indices $p, q\in \mathbb{Z}_{\ge 0}$ and $\{b_i\}_{i = 1}^{p+q}$, $\{d_i\}_{i = 1}^p\subset\R_{\ge 0}$ such that  $b_i < d_i$ for all $i=1, \ldots, p$. Furthermore, the following graded module isomorphism holds. 
\begin{equation}\label{eq:structure}
\ph_k(\cX) \simeq \bigoplus_{i=1}^p\left((z^{b_i})/ (z^{d_i})\right)\oplus\bigoplus_{i=p+1}^{p+q}(z^{b_i}), 
\end{equation}
where $(z^{a})$ expresses an ideal in $\R[\R_{\ge 0}]$ generated by the monomial $z^a$. 
\end{theorem}
Here, $b_i$ and $d_i$ are called the $k$-th birth and death times, respectively, and they indicate the times of the appearance and disappearance (again, respectively) of $k$-dimensional holes in the filtration $\cX = \{X(t)\}_{t\ge 0}$. The lifetime $l_i$ is defined by $l_i := d_i - b_i$. We set $d_i = \infty$ for $p+1 \le i \le p+q$. 
We define the lifetime sum $L_k(\cX)$ by
\[
L_k(\cX)=\sum_{i=1}^{p+q}(d_i-b_i).
\]
The following formula is a generalization of \eqref{eq: key of Frieze's zeta 3} to filtrations. 
\begin{theorem}[Lifetime formula {\cite[Proposition~2.2]{HS}}]	\label{thm:lifetimeformula}
Let $\cX = \{X(t)\}_{t\ge 0}$ be a right-continuous filtration of a simplicial complex $X$.
Then, for each $k\ge0$, 
\[
L_k(\cX) = \int_0^\infty \b_k(X(t))\,dt. 
\]
\end{theorem}
Let $k\ge0$ and $T\ge0$. We also define the $k$-th lifetime sum until time $T$ by 
\[
(L_k(\cX))_T=\sum_{i=1}^{p+q}((d_i\wedge T) -(b_i\wedge T)).
\]
As an analogue of Theorem~\ref{thm:lifetimeformula}, the equality 
\[
(L_k(\cX))_T = \int_0^T \b_k(X(t))\,dt 
\]
holds.
\subsection{Lifetime sums of random simplicial complex processes}
We consider a class of random simplicial complex processes associated with an $n$-point vertex set $V$ and moving multi-parameters $\p = (p_0, p_1, \ldots, p_{n-1})$. 
Let $n\in\N$ and denote the complete $(n-1)$-dimensional simplicial complex by $K(n)$, that is, the family of all nonempty subsets of $V$. 
To each simplex $\sg\in K(n)$ we assign an independent non-negative random variable $u_\sg$ with distribution function $F_\sg$. We assume that these distribution functions $F_\sg$ are identical for equal-dimensional simplices and, for each $\sg\in K(n)$ with $\dim\sg=i$, we denote $F_\sg$ by $p_i(\cdot)$. 
We define a random simplicial complex process $\cX_n = \{X_n(t)\}_{t\ge0}$ by
\[
X_n(t) := \{\sg\in K(n)\mid u_\tau \le t \text{ for every simplex $\emptyset\neq\tau\subset\sigma$}\}. 
\]
We call this process the multi-parameter random complex process with $n$ vertices and multi-parameter function $\p(\cdot) = (p_0(\cdot), p_1(\cdot), \ldots, p_{n-1}(\cdot))$. 
Note that $X_n(t) \sim X(n,\p(t))$ for fixed $t$ and that $\cX_n$ is a right-continuous filtration of $ K(n)$. 
Note also that $X_n(t)$ can be expressed as 
\[
X_n(t)=\{\sg\in K(n)\mid w_\sg\le t\}, 
\]
where $w_\sg:=\max\{u_\sg\mid\emptyset\neq\tau\subset\sg\}$. 

In what follows, each $p_i(t)$ is assumed to be independent of $n$.
We write $q_k(t)$ and $r_k(t)$, respectively, for the $q_k$ and $r_k$ defined in \eqref{eq:parameters} that are associated with $X_n(t)$.
The generalized inverse function $\check r_k$ of $r_k$ is defined by
\[
\check r_k(u)=\inf\{t\ge0\mid r_k(t)>u\}\quad \text{for }u<1,
\]
and $\check r_k(1)=\infty$.
We additionally define
\begin{gather*}
Q_k(t)=\int_0^t q_k(s)\,ds\quad\text{for }t\ge0,\\
\Phi_k(u)=Q_k(\check r_k(u))\text{ and }\Psi_k(u)=Q_k(\check r_{k-1}(u))\quad \text{for }u\in[0,1).
\end{gather*}
We note that $\Phi_k\ge\Psi_k$ since $\check r_k$ is nondecreasing with respect to $k$. 
Moreover, $\check r_k$, $\Phi_k$, and $\Psi_k$ are nondecreasing right-continuous functions.
The following relations are fundamental. For $t\ge0$, $u\in[0,1)$, and $\eps>0$,
\begin{itemize}
\item $\check r_k(r_k(t)-\eps)\le t\le \check r_k(r_k(t))$; and
\item $r_k(\check r_k(u)-\eps)\le u\le r_k(\check r_k(u))$ if $\check r_k(u)\ge\eps$. 
\end{itemize}
We provide results for the asymptotic behavior of the lifetime sum of $\cX_n$ in the following theorems.
\begin{theorem}	\label{thm:CForder}
\begin{enumerate}
\item If there exist $A\in(0,1]$ and $u_0\in(0,1)$ such that
\begin{equation}\label{eq:CForder_Phi}
\Phi_k(u_0)>0\quad\text{and}\quad\Phi_k(u/2)\ge A\Phi_k(u)\text{ for }0\le u\le u_0,
\end{equation}
then, for any $m\in\N$, there exists a constant $C\ge0$ such that for sufficiently large $n$, 
\[
\E[(L_k(\cX_n))_T] \le C n^{k+1}\Phi_k(1/n)(1+Tn^{-m})\quad\text{for $T>0$. }
\]
Moreover, if $\int_0^\infty t^{1+\dl}\,dq_{k+1}(t)<\infty$ for some $\dl>0$, then 
\[
\E[L_k(\cX_n)] =O(n^{k+1}\Phi_k(1/n)). 
\]
\item If $\Phi_k(u)=O(u^m)$ as $u\to0$ for all $m\ge0$, 
then, for all $m\in\N$, there exists a constant $C\ge0$ such that for sufficiently large $n$, 
\[
\E[(L_k(\cX_n))_T]\le C(1+T)n^{-m}\quad\text{for $T>0$. }
\]
Moreover, if $\int_0^\infty t^{1+\dl}\,dq_{k+1}(t)<\infty$ for some $\dl>0$, then $\E[L_k(\cX_n)] =O(n^{-m})$ for all $m\ge0$.
\end{enumerate}
\end{theorem}
\begin{theorem}	\label{thm:CForder2}
Suppose that there exist $u_0\in(0,1)$, $B>1$, and $D\in[0,(k+2)/(4(k+1)))$ such that
\[
\Phi_k(Du) \ge B \Psi_k(u)\quad\text{for }0\le u\le u_0.
\]
Then, there exists some $c>0$ such that for $T>\check r_k(0)$, 
\[
\E[(L_k(\cX_n))_T] \ge c n^{k+1}\Phi_k(1/n)\quad\text{for sufficiently large $n$.}
\]
In particular, $\E[L_k(\cX_n)]=\Om(n^{k+1}\Phi_k(1/n))$. 
\end{theorem}
Before proceeding to the proof, we consider typical situations in which
\begin{equation}\label{eq:typical_example}
\Phi_k(u)=\Theta(u^a)\text{ as }u\to0
\end{equation}
for some constant $0\le a\le\infty$. 
Here, we write $f(u)=\Theta(g(u))$ as $u\to0$ to indicate that there exist $c_1>0$, $c_2>0$, and $u_0\in(0,1)$ such that $c_1 g(u)\le f(u)\le c_2 g(u)$ for all $u\in(0,u_0]$.
Note also that $u^\infty=0$ and $0^0=1$ by convention. 
Then, we have the following. 
\begin{itemize}
\item Theorem~\ref{thm:CForder}(1) holds when $a<\infty$; 
\item Theorem~\ref{thm:CForder}(2) holds when $a=\infty$; 
\item Theorem~\ref{thm:CForder2} holds when $\Psi_k(u)=o(\Phi_k(u))\text{ as }u\to0$. 
\end{itemize}
In particular, we have the following result.
\begin{cor}\label{cor:main}
Suppose that $\Phi_k$ satisfies \eqref{eq:typical_example} with $0\le a<\infty$ and $\Psi_k(u)=o(\Phi_k(u))\text{ as }u\to0$. Then, for each $T>0$, 
\[
\E[(L_k(\cX_n))_T]\asymp n^{k+1-a}.
\]
Moreover, if $\int_0^\infty t^{1+\dl}\,dq_{k+1}(t)<\infty$ for some $\dl>0$, then $\E[L_k(\cX_n)]\asymp n^{k+1-a}$.
\end{cor}
The case $\Psi_k(u)=\Theta(\Phi_k(u))\text{ as }u\to0$ is sensitive, and we cannot provide a simple answer for it because 
more detailed relations between $\Phi_k$ and $\Psi_k$ affect the asymptotic behaviors of $\E[(L_k(\cX_n))_T]$ and $\E[L_k(\cX_n)]$, and Theorem~\ref{thm:CForder} does not always give proper upper estimates. Here, we just consider the case where $\Psi_k(u)$ is identically $\Phi_k(u)$ for all $u\in[0,1)$. 
\begin{theorem}\label{thm:delicate_case}
If $\Phi_k(u)=\Psi_k(u)$ for all $u\in[0,1)$, then $L_k(\cX_n)=0$ almost surely for all $n\in\N$. 
\end{theorem}
Theorems~\ref{thm:HS1} and \ref{thm: order of 1-flag} are  special cases of Corollary~\ref{cor:main}, as will be shown below. 
We can modify the range of the parameter $t$ to $[0,\infty)$ by setting $p_i(t)=1$ for $t\ge1$ and all $i$: in other words, $X_n(t)=K(n)$ for $t\ge1$. This modification does not affect the lifetimes because the one-point set $\{1\}$ is a Lebesgue null set and all dimensional homologies of $K(n)$ vanish. When we are interested in a filtration with parameters in a finite interval, we will make such a modification, if necessary, without explicitly mentioning it.
\begin{example}\label{ex:LMprocess}
Let $n>d \ge 1$ be fixed and define $\p(t) = (p_0(t), \ldots, p_{n-1}(t))$ by
\[
p_i(t) := \begin{cases}
1		& (0\le i \le d-1), \\
t 		& (i = d), \\
0 		& (d+1 \le i \le n-1)
\end{cases}
\quad\text{for }0\le t\le1.
\]
The corresponding process $\cK_n^{(d)} = \{K_n^{(d)}(t)\}_{0 \le t \le 1}$ is called the $d$-Linial--Meshulam complex process. This is a higher-dimensional analogue of the Erd\H{o}s--R\'{e}nyi graph process $\cK_n=\{K_n(t)\}_{0\le t\le1}$, which is identified with the $1$-Linial--Meshulam complex process. 
We can easily confirm that
\[
(\Phi_k(u),\Psi_k(u))=\begin{cases}
(0,0)	& (k<d-1), \\
(u,0)=(u^1,u^\infty)  & (k = d-1), \\
(1/2,u^2/2)=(\Theta(u^0),\Theta(u^2))	& (k = d), \\
(0,0)	& (k>d). 
\end{cases}
\]
From Corollary~\ref{cor:main} and Theorem~\ref{thm:delicate_case}, we have
\[
\E[L_k(\cK_n^{(d)})] \asymp \begin{cases}
0				& (k \neq d-1, d), \\
n^{d-1} 		& (k = d-1), \\
n^{d+1} 		& (k = d). 
\end{cases}
\]
The case $k=d-1$ corresponds to Theorem~\ref{thm:HS1}.
\end{example}
\begin{example}\label{ex:CLprocess}
Let $n>d \ge 1$ be fixed and define $\p(t) = (p_0(t), \ldots, p_{n-1}(t))$ by
\[
p_i(t) := \begin{cases}
1		& (0\le i \le d-1), \\
t 		& (i = d), \\
1 		& (d+1 \le i \le n-1)
\end{cases}
\quad\text{for }0\le t\le1.
\]
We call the corresponding process $\cC_n^{(d)} = \{C_n^{(d)}(t)\}_{0 \le t \le 1}$ the $d$-flag complex process. 
Note that the case $d=1$ corresponds to the random clique complex process $\cC_n = \{C_n(t)\}_{0 \le t \le 1}$.

By straightforward computation, we have
\[
(\Phi_k(u),\Psi_k(u))=\begin{cases}
(0,0)	& (k<d-1), \\
(u,0)=(u^1,u^\infty)	& (k=d-1), \\
\biggl(\Theta\Bigl(u^{\frac{k+1-d}{d+1}+\binom{k+1}d^{-1}}\Bigr),\Theta\Bigl(u^{\frac{k+1}{d+1}+\binom kd^{-1}}\Bigr)\biggr) 	& (k \ge d). 
\end{cases}
\]
From Corollary~\ref{cor:main} and Theorem~\ref{thm:delicate_case}, we have
\[
\E[L_k(\cC^{(d)}_n)] \asymp \begin{cases}
0 & (k<d-1),\\
n^{\frac{(k+2)d}{d+1}-\binom{k+1}{d}^{-1}} &(k\ge d-1).
\end{cases}
\]
The case $d=1$ corresponds to Theorem~\ref{thm: order of 1-flag}.
\end{example}
\begin{proof}[Proof of Theorem~\ref{thm:CForder}]
(1) From \eqref{eq:CForder_Phi}, for any $j\in\N$,
\[
\Phi_k(u/2^j)\ge A^j\Phi_k(u)
\quad\text{for }u\in[0,u_0].
\]
Thus, the following estimate holds for $0\le K\le 1$.
\begin{equation}\label{eq:CForder_estimate1}
\Phi_k(K u) \ge A K^{\gm}\Phi_k(u)
\quad\text{for }u\in[0,u_0],
\end{equation}
where $\gm=-\log_2 A\ge0$.
In particular, we have
\begin{equation}\label{eq:CForder_estimate2}
\Phi_k(v)\ge A\Phi_k(u_0)v^{\gm} 
\quad\text{for }v\in[0,u_0] 
\end{equation}
by letting $u=u_0$ and $K=v/u_0$ in \eqref{eq:CForder_estimate1}. 
Let $m\in\N$. Take $l\in\N$ such that $l\ge\gm+m$.
Theorem~\ref{thm:CFdecayBetti} implies that there exists some $C\ge0$ such that for all $n$,
\[
\E[\b_k(X_n(t))]
\le n^{k+1}q_k(t)\{1\wedge C(nr_k(t))^{-l}\}.
\]
Let $T>0$. In what follows, $n\in\N$ is taken to be sufficiently large and independent of $T$. 
Define $S_n=\check r_k(1/n)$.
Then, 
\[
\int_0^{S_n\wedge T} \E[\b_k(X_n(t))]\,dt\le n^{k+1}\Phi_k(1/n). 
\]
Next, suppose $t> S_n$. Then, $r_k(t)\ge r_k(S_n)\ge1/n$. Let $\Xi(t)=-(nr_k(t))^{-l}$, which is a right-continuous nondecreasing function. 
Then, denoting $\check r_k(u_0)$ by $t_0$, 
\begin{align*}
&\int_{S_n\wedge T}^T\E[\b_k(X_n(t))]\,dt\\
&=\left(\int_{(S_n\wedge T,t_0\wedge T]}+\int_{(t_0\wedge T,T]}\right)\E[\b_k(X_n(t))]\,dt\\
&\le -Cn^{k+1}\int_{(S_n,t_0]} Q_k'(t)\Xi(t)\,dt+T C n^{k+1}q_k(T)(n r_k(t_0))^{-l}\\
&=-C n^{k+1}\left(\left[Q_k(t)\Xi(t)\right]_{t=S_n}^{t=t_0}-\int_{(S_n,t_0]} Q_k(t)\,d\Xi(t)\right)+T C n^{k+1}q_k(T)(n r_k(t_0))^{-l}\\
&\le C n^{k+1}\left(\Phi_k(u_0)(n u_0)^{-l}+\lim_{\eps\downarrow0}\int_{(S_n+\eps,t_0]} Q_k(t-\eps)\,d\Xi(t)+ T (n u_0)^{-l} \right).
\end{align*}
By noting that $\Phi_k(u_0)>0$, 
from \eqref{eq:CForder_estimate2}, we have
\[
(\Phi_k(u_0)+T)(n u_0)^{-l}\le\frac{\Phi_k(u_0)+T}{Au_0^l\Phi_k(u_0)}n^{-m}\Phi_k(1/n) \le C' (1+T n^{-m})\Phi_k(1/n), 
\]
where $C'>0$ is a constant independent of $n$ and $T$. 
Writing $\check\Xi(u)=\inf\{t\ge0\mid \Xi(t)>u\}$ for $u\in\R$ and taking a small $\eps>0$, we have
\begin{align*}
\int_{(S_n+\eps,t_0]} Q_k(t-\eps)\,d\Xi(t)
&\le\int_{(S_n+\eps,t_0]} \Phi_k(r_k(t-\eps))\,d\Xi(t)\\
&\le \int_{(S_n+\eps,t_0]}\frac1A (nr_k(t-\eps))^{\gm}\Phi_k(1/n)\,d\Xi(t)\\
&= \frac1A \Phi_k(1/n)\int_{(S_n+\eps,t_0]}(-\Xi(t-\eps))^{-\gm/l}\,d\Xi(t)\\
&= \frac1A \Phi_k(1/n)\int_{(\Xi(S_n+\eps),\Xi(t_0)]}(-\Xi(\check\Xi(u)-\eps))^{-\gm/l}\,du\\
&\le \frac1A \Phi_k(1/n)\int_{\Xi(S_n+\eps)}^{\Xi(t_0)}(-u)^{-\gm/l}\,du\\
&=\frac1A \Phi_k(1/n)\cdot\frac{-1}{1-\gm/l}\{(-\Xi(t_0))^{1-\gm/l}-(-\Xi(S_n+\eps))^{1-\gm/l}\}\\
&\le \frac1{A(1-\gm/l)}\Phi_k(1/n)(n r_k(S_n))^{\gm-l}\\
&\le \frac1{A(1-\gm/l)}\Phi_k(1/n).
\end{align*}
Here, we used \eqref{eq:CForder_estimate1} with $K=1/(n r_k(t-\eps))$ and $u=r_k(t-\eps)$ in the second line, noting that $K\le1$ and $u\le u_0$ for $t\in(S_n+\eps,t_0]$. In the fourth line, we used the change of variable formula with $t=\check\Xi(u)$.
Thus, 
\[
\int_{S_n\wedge T}^T\E[\b_k(X_n(t))]\,dt\le C\left(C'+\frac1{A(1-\gm/l)}\right)n^{k+1}\Phi_k(1/n)(1+T n^{-m})
\]
and the first result given in Theorem~\ref{thm:CForder}~(1) follows.

Now, suppose that $M:=\int_0^\infty t^{1+\dl}\,dq_{k+1}(t)<\infty$ for some $\dl>0$. 
Let $w_\sg=\max\{u_\tau\mid\emptyset\neq\tau\subset\sg\}$ for $\sg\in K(n)$ and define $U_n=\max\left\{w_\sg\relmiddle|\sg\in\binom V{k+2}\right\}$. The distribution function of $w_\sg$ for $\sg\in\binom V{k+2}$ is equal to $q_{k+1}(\cdot)$. 
Take $m\in\N$ such that $m\ge(\gm+k+2)/\dl$. 
From the first conclusion with this $m$, there exists a constant $C\ge0$ such that, for sufficiently large $n$, 
\[
\E[(L_k(\cX_n))_{n^m}] \le C n^{k+1}\Phi_k(1/n)
\]
by letting $T=n^m$. Then, 
\begin{align*}
\E[L_k(\cX_n)]
&= \E\biggl[\int_0^{U_n}\b_k(X_n(t))\,dt\biggr]\\
&= \E\biggl[\int_0^{U_n}\b_k(X_n(t))\,dt\,; U_n\le n^m\biggr]+\E\biggl[\int_0^{U_n}\b_k(X_n(t))\,dt\,; U_n> n^m\biggr]\\
&\le \E[(L_k(\cX_n))_{n^m}]+n^{k+1}\E[U_n\,; U_n> n^m]\\
&\le C n^{k+1}\Phi_k(1/n)+n^{k+1}n^{-m\dl}\E[U_n^{1+\dl}]. 
\end{align*}
For the second term, we have
\begin{align*}
n^{-m\dl}\E[U_n^{1+\dl}]
&\le n^{-m\dl}\E\left[\sum_{\sg\in\binom V{k+2}}w_\sg^{1+\dl}\right]\\
&\le n^{-(\gm+k+2)} \binom n{k+2}\int_0^\infty t^{1+\dl}\,dq_{k+1}(t)\\
&\le M n^{-\gm}\\
&\le \frac M{A\Phi_k(u_0)}\Phi_k(1/n). \quad\text{(from \eqref{eq:CForder_estimate2})}
\end{align*}
The final conclusion in part (1) of the theorem follows by combining these estimates. 

(2) Let $l\in\N$ and $\eps_n=n^{-1/2}$. From the assumption, there exist $c>0$ and $u_0\in(0,1)$ such that $\Phi_k(u)\le c u^l$ for all $u\in(0,u_0]$. Moreover, $r_k(t)\ge \eps_n$ for $t\ge \check r_k(\eps_n)$. 
Thus, from Theorem~\ref{thm:CFdecayBetti}, for sufficiently large $n$, 
\begin{align*}
\E[(L_k(\cX_n))_T]
&=\int_0^{\check r_k(\eps_n)\wedge T} \E[\b_k(X_n(t))]\,dt + \int_{\check r_k(\eps_n)\wedge T}^T \E[\b_k(X_n(t))]\,dt\\
&\le \binom n{k+1}\Phi_k(\eps_n) + C n^{k+1}\int_{\check r_k(\eps_n)\wedge T}^T q_k(t)(nr_k(t))^{-l}\,dt\\
&\le c n^{k+1-l/2} + C n^{k+1-l/2}\int_0^T q_k(t)\,dt\\
&= (c+CT)n^{k+1-l/2}
\end{align*}
for $T>0$. 

The second conclusion of part (2) of the theorem follows in the same manner as that of part (1).  
\end{proof}
\begin{proof}[Proof of Theorem~\ref{thm:CForder2}]
Let $T>\check r_k(0)$ and $\a=k/2 +1$. Define
\[
\tilde S_n=\check r_{k-1}(\a/(Dn))\text{ and }
\tilde T_n=\check r_k(\a/n)
\]
for large enough $n$ such that $\a/(Dn)\le u_0$ and $\tilde T_n\le T$. 

Suppose $\tilde S_n\ge \tilde T_n$ for some $n$. 
Then, $\Psi_k(\a/(Dn))=Q_k(\tilde S_n)\ge Q_k(\tilde T_n)=\Phi_k(\a/n)$, while $\Phi_k(\a/n)\ge B \Psi_k(\a/(Dn))$ by assumption.
Therefore, $\Phi_k(\a/n)=0$, which implies $\Phi_k(u)=0$ for $u\in[0,\a/n]$. In this case, the conclusion is trivially true.

Thus, we may assume that $\tilde S_n<\tilde T_n$ for every $n$. If $t\in[\tilde S_n,\tilde T_n)$, then
\[
r_{k-1}(t)\ge r_{k-1}(\tilde S_n) \ge \a/(Dn)
\]
and 
\[
r_k(t)\le \lim_{\eps\downarrow0}r_k(\tilde T_n-\eps)\le \a/n.
\]
Since
\[
\frac{(k+1)D}\a+\frac\a{k+2}=\frac{2(k+1)D}{k+2}+\frac12<1,
\]
we can apply Proposition~\ref{prop:byMorse} to obtain the existence of $n_0\in\N$ and $\eps_0>0$ such that, if $n\ge n_0$, then
\[
\E[\b_k(X_n(t))]\ge\eps_0 n^{k+1}q_k(t)\quad\text{for $t\in[\tilde S_n,\tilde T_n)$. }
\]
Thus, for $n\ge n_0$, 
\begin{align*}
\E[(L_k(\cX_n))_T]
&\ge\int_{\tilde S_n}^{\tilde T_n}\E[\b_k(X_n(t))]\,dt\\
&\ge\eps_0 n^{k+1}\int_{\tilde S_n}^{\tilde T_n}q_k(t)\,dt\\
&=\eps_0 n^{k+1}(\Phi_k(\a/n)-\Psi_k(\a/(Dn)))\\
&\ge\eps_0 n^{k+1}\left(\Phi_k(\a/n)-B^{-1}\Phi_k(\a/n)\right)\\
&\ge\eps_0(1-B^{-1})n^{k+1}\Phi_k(1/n).\myqedhere
\end{align*}
\end{proof}
\begin{proof}[Proof of Theorem~\ref{thm:delicate_case}]
Note that for any $0\le k\le n-1$ and $t\ge0$, 
\[
q_k(t)\ge\prod_{i=0}^{k} (p_i(t))^{(i+1)\binom{k+1}{i+1}}=\prod_{i=0}^{k} (p_i(t))^{(k+1)\binom ki}=(r_{k-1}(t))^{k+1}. 
\]
Then, for $u\in(0,1)$ and $t\ge\check r_{k-1}(u)$, the inequality $q_k(t)\ge (r_{k-1}(t))^{k+1}\ge u^{k+1}$ holds.
Thus, 
\[
0=\Phi_k(u)-\Psi_k(u)=\int_{\check r_{k-1}(u)}^{\check r_k(u)}q_k(t)\,dt
\ge u^{k+1}(\check r_k(u)-\check r_{k-1}(u)). 
\]
Therefore, 
$\check r_{k-1}(u)= \check r_k(u)$ for $u\in(0,1)$, which implies that $r_{k-1}(t)= r_k(t)$ for $t\ge0$ from the (weak) monotonicity and right-continuity of $r_k$. 
Therefore, for $t\ge0$, 
\begin{equation}\label{eq:delicate_case}
0=r_{k-1}(t)-r_k(t)=r_{k-1}(t)\left(1-\prod_{i=1}^{k+1}p_i(t)^{\binom k{i-1}}\right)\ge q_k(t)\left(1-\prod_{i=1}^{k+1}p_i(t)^{\binom k{i-1}}\right). 
\end{equation}
Here, in the inequality above, we used the relation $r_{k-1}(t)\ge r_{k-1}(t)q_{k-1}(t)=q_k(t)$. 
Let $t\ge0$. 
\eqref{eq:delicate_case} implies that $q_k(t)=0$ or $p_i(t)=1$ for all $i=1,2,\dots,k+1$. If $q_k(t)=0$, then $\b_k(X_n(t))=0$ almost surely. Suppose that $p_i(t)=1$ for all $i=1,2,\dots,k+1$.
Then, $\b_k(X_n(t))=0$ almost surely since $X_n(t)$ includes the complete $(k+1)$-dimensional skeleton. 
Thus, in all cases, $\b_k(X_n(t))=0$ almost surely. This implies $L_k(\cX_n)=0$ a.s.
\end{proof}
\subsection{Limiting constants}
As a refinement of Theorems~\ref{thm:CForder} and \ref{thm:CForder2}, 
it is natural to consider the behavior of the normalized $k$-th lifetime sum
\[
\overline{L}_k(\cX_n) := \frac{L_k(\cX_n)}{n^{k+1}\Phi_k(1/n)}
\]
of a filtration $\cX_n$. It is not yet known what general conditions are needed for $\overline{L}_k(\cX_n)$ to converge in a certain sense as $n\to\infty$. In the case where $\cX_n$ is the $d$-Linial--Meshulam complex process $\cK_n^{(d)}=\{K_n^{(d)}(t)\}_{0\le t\le1}$, Hiraoka and Shirai~\cite[Section~7.1]{HS} made a formal argument and conjectured that the expectation of $\overline{L}_{d-1}(\cK_n^{(d)})$ converges to some positive constant $I_{d-1}$. 
Their argument was based on recent work by Linial and Peled~\cite{LP} on the convergence of $K_n^{(d)}(c/n)$ for fixed $c \ge 0$. 
We justify their argument and prove Theorem~\ref{thm:LMlimiting} in a more general form by using the upper estimate in Example~\ref{ex:LM}. 

One of the special features of $\ph_{d-1}(\cK_n^{(d)})$ is that all the birth times $b_i$ in \eqref{eq:structure} are zero because $K_n^{(d)}(0)$ is the complete $(d-1)$-dimensional simplicial complex.
Given this, we can obtain a formula for the generalized sums of lifetimes as follows. 
Let $d_i$ $(i=1,\dots,p+q)$ be the death times in \eqref{eq:structure} for the $k$-th persistent homology of a general filtration $\cX=\{X(t)\}_{t\ge0}$.
\begin{prop}\label{prop:lifetime_formula2}
Let $\varphi$ be a right-continuous nondecreasing function on $[0,\infty)$ with $\varphi(0)=0$. Suppose that $b_i=0$ for $i=1,\ldots,p+q$. Then,
\[
\sum_{i=1}^{p+q}\varphi(d_i-)=\int_{[0,\infty)} \b_k(X(t))\,d \varphi(t),
\]
where $\varphi(t-)=\lim_{\eps\downarrow0}\varphi(t-\eps)$.
\end{prop}
\begin{proof}
This is proved by simple calculation:
\begin{align*}
\sum_{i=1}^{p+q}\varphi(d_i-)
&=\sum_{i=1}^{p+q}\int_0^\infty 1_{[0,d_i)}(t)\,d\varphi(t)\\
&=\int_0^\infty\left(\sum_{i=1}^{p+q}1_{[0,d_i)}(t)\right)d\varphi(t)\\
&=\int_0^\infty\b_k(X(t))\,d\varphi(t).\myqedhere
\end{align*}
\end{proof}
Let $\a>0$ and $d\in\N$. We consider the $d$-Linial--Meshulam complex process $\cK_n^{(d)}$ and define
\[
  L_{d-1}^{(\a)}(\cK_n^{(d)})=\sum_{i=1}^{p+q}{d_i}^\a,
\]
which is the sum of the $\a$-th powers of the $(d-1)$-th lifetimes of $\cK_n^{(d)}$. Clearly, $L_{d-1}^{(1)}(\cK_n^{(d)})=L_{d-1}(\cK_n^{(d)})$. From Proposition~\ref{prop:lifetime_formula2}, 
\[
L_{d-1}^{(\a)}(\cK_n^{(d)})=\a\int_0^1 \b_{d-1}(K_n^{(d)}(t)) t^{\a-1}\,dt.
\]
Below, we study the precise asymptotic behavior of $L_{d-1}^{(\a)}(\cK_n^{(d)})$ as $n\to\infty$.

We recall some results in~\cite{LP}. 
For $d \ge 2$, let $t_d^*$ be the unique root in $(0, 1)$ of the equation
\[
(d+1)(1-t) + (1+dt)\log t = 0, 
\]
and define the constant $c_d^* = \psi_d(t_d^*)>0$, 
where
\[
\psi_d(t) = \frac{-\log t}{(1-t)^d},\quad t\in (0, 1). 
\] 
For $d=1$, define $t_1^* = c_1^* = 1$. For $c \ge c_d^*$, let $t_c$ denote the smallest positive root of the equation $\psi_d(t) = c$. Note that $t_c \le t_d^*$. 
\begin{theorem}
For $c\ge0$,
\begin{equation}\label{eq:LM_E}
\lim_{n\to\infty}\frac{\E[\b_d(K_n^{(d)}(c/n))]}{\binom{n}{d}}=g_d(c), 
\end{equation}
where
\[
g_d(c)=
\begin{cases}
0		&(c < c_d^*), \\
ct_c(1-t_c)^d + \frac{c}{d+1}(1-t_c)^{d+1} - (1-t_c)	&(c \ge c_d^*).
\end{cases}	
\]
Moreover, for any $\eps>0$, 
\begin{equation}\label{eq:LM_P}
\lim_{n\to\infty}\P\left[\Bigl|\frac{\b_d(K_n^{(d)}(c/n))}{\binom{n}{d}}-g_d(c)\Bigr|>\eps\right]=0. 
\end{equation}
\end{theorem}
\begin{proof}
The claim for $c\ne c_d^*$ follows from the results in \cite{LP}.
When $c=c_d^*$, the assertion follows from the monotonicity of $\b_d(K_n^{(d)}(\cdot))$ and the continuity of $g_d$ with $g_d(c_d^*)=0$.
Indeed, for $\eps>0$, take $c'>c_d^*$ such that $g_d(c')\le\eps/2$. Then, 
\begin{align*}
\P\left[\frac{\b_d(K_n^{(d)}(c_d^*/n))}{\binom{n}{d}}>\eps\right]
&\le \P\left[\frac{\b_d(K_n^{(d)}(c_d^*/n))}{\binom{n}{d}}>\eps/2+g_d(c')\right]\\
&\le \P\left[\frac{\b_d(K_n^{(d)}(c'/n))}{\binom{n}{d}}-g_d(c')>\eps/2\right]. 
\end{align*}
Since the last term converges to $0$ as $n\to\infty$, we obtain \eqref{eq:LM_P} for $c=c_d^*$. 
The proof of \eqref{eq:LM_E} with $c=c_d^*$ is similar.
\end{proof}
Since
\[
\b_d(K_n^{(d)}(t)) - \b_{d-1}(K_n^{(d)}(t)) = f_d(K_n^{(d)}(t)) - \binom{n-1}{d}
\]
by the Euler--Poincar\'{e} formula, we have
\[
\frac{\b_{d-1}(K_n^{(d)}(c/n))}{\binom{n}{d}} = \frac{1}{\binom{n}{d}}\left(\b_d(K_n^{(d)}(c/n)) + \binom{n-1}{d} - f_d(K_n^{(d)}(c/n))\right). 
\]
Write
\[
Z_n := \frac{f_d(K_n^{(d)}(c/n))}{\binom{n}{d}}. 
\]
Then, $\E[Z_n]=\frac{c}{d+1}(1-d/n)$. Note that $f_d(K_n^{(d)}(c/n))\sim\bin\bigl(\binom n{d+1}, c/n\bigr)$ for the $d$-Linial--Meshulam complex. 
Thus, by direct computation, 
\[
\E[(Z_n-\E[Z_n])^2]=\frac c{d+1}\left(1-\frac dn\right)\left(1-\frac cn\right)\binom n d^{-1}\to0 \text{ as }n\to\infty. 
\]
Therefore, we obtain $\lim_{n\to\infty}\E[(Z_n-c/(d+1))^2]=0$.

From these estimates, for each $c\ge0$,
\begin{equation*}
\lim_{n\to\infty}\frac{\E\bigl[\b_{d-1}(K_n^{(d)}(c/n))\bigr]}{\binom{n}{d}} = h_d(c)
\end{equation*}
and, for any $\eps>0$, 
\begin{equation}\label{eq:d-1_Betti_in_pr}
\lim_{n\to\infty}\P\left[\Bigl|\frac{\b_{d-1}(K_n^{(d)}(c/n))}{\binom{n}{d}}-h_d(c)\Bigr|>\eps\right]=0, 
\end{equation}
where
\begin{equation}\label{eq:hdc}
h_d(c):=1 - \frac{c}{d+1} + g_d(c)\ge0. 
\end{equation}
For $\a>0$, define
\begin{equation}\label{eq:I}
I_{d-1}^{(\a)}=\frac{\a}{d!}\int_0^\infty h_d(s)s^{\a-1}\,ds.
\end{equation}
The constant in Theorem~\ref{thm:LMlimiting} is then defined by $I_{d-1}:=I_{d-1}^{(1)}$.
The following theorem is the main result of this subsection. Theorem~\ref{thm:LMlimiting} is a particular case of Theorem~\ref{thm:LMlimiting2} with $\a=1$.
\begin{theorem}\label{thm:LMlimiting2}
Let $d\ge1$ and $\a>0$. Then $I_{d-1}^{(\a)}$ is finite, and for any $r \in[1,\infty)$,
\[
\lim_{n\to\infty}\E\left[\left|\frac{L_{d-1}^{(\a)}(\cK_n^{(d)})}{n^{d-\a}} - I_{d-1}^{(\a)}\right|^r\right]=0.
\]
In particular, $\E[L_{d-1}^{(\a)}(\cK_n^{(d)})]/n^{d-\a}$ converges to $I_{d-1}^{(\a)}$ as $n\to\infty$.
\end{theorem}
\begin{proof}
We may assume without loss of generality that all random variables are defined in a common probability space $(\Om,\cF,\P)$.
Denote the $L^r$-norm on $(\Om,\cF,\P)$ by $\|\cdot\|_{L^r}$.
From the second inequality of \eqref{eq:MorseIneq} and \eqref{eq:d-1_Betti_in_pr}, for $s\ge0$,  we have
\begin{equation}\label{eq:LMlimiting2_1}
0\le \frac{\b_{d-1}(K_n^{(d)}(s/n))}{n^d}1_{[0,n]}(s)\le \frac1{d!}
\quad\text{for }n\in\N
\end{equation}
and
\begin{equation}\label{eq:LMlimiting2_conv}
\frac{\b_{d-1}(K_n^{(d)}(s/n))}{n^d}1_{[0,n]}(s)\to \frac{1}{d!}h_d(s)\text{ in probability as $n\to\infty$}.
\end{equation}
Take $l\in\N$ such that $l>1\vee r\a$. From Example~\ref{ex:LM}~(1), there exists $C\ge0$ such that
\begin{equation}\label{eq:LMlimiting2_2}
\sup_{n\in\N}\E\left[\frac{\b_{d-1}(K_n^{(d)}(s/n))}{n^d}1_{[0,n]}(s)\right]\le 1\wedge C s^{-l}\quad\text{for $s\ge0$. }
\end{equation}
Applying Fatou's lemma to an appropriate subsequence, we obtain $\frac{1}{d!}h_d(s)\le 1\wedge C s^{-l}$.
From this estimate, $I_{d-1}^{(\a)}$ must be finite.
From Proposition~\ref{prop:lifetime_formula2} and Minkowski's inequality, 
\begin{align*}
&\left\|L_{d-1}^{(\a)}(\cK_n^{(d)})/n^{d-\a}-I_{d-1}^{(\a)}\right\|_{L^r}\\
&= \left\|\a\int_0^{\infty}\left(\frac{\b_{d-1}(K_n^{(d)}(s/n))}{n^{d}}1_{[0,n]}(s)-\frac{1}{d!}h_d(s)\right)s^{\a-1}ds\right\|_{L^r}\\
&\le \a\int_0^{\infty}U_n(s) s^{\a-1}\,ds, 
\end{align*}
where $U_n(s)=\left\|\frac{\b_{d-1}(K_n^{(d)}(s/n))}{n^{d}}1_{[0,n]}(s)-\frac{1}{d!}h_d(s)\right\|_{L^r}$.
Combining \eqref{eq:LMlimiting2_1} and \eqref{eq:LMlimiting2_conv}, we obtain $\lim_{n\to\infty}U_n(s)= 0$ for each $s\ge0$. 
Moreover, from \eqref{eq:LMlimiting2_1} and \eqref{eq:LMlimiting2_2},
\begin{align*}
\sup_{n\in\N}U_n(s)
&\le \sup_{n\ge s}\E\left[ \left(\frac{1}{d!}\right)^{r-1}\frac{\b_{d-1}(K_n^{(d)}(s/n))}{n^d}\right]^{1/r}+\frac{1}{d!}h_d(s)\\
&\le \left(\frac{1}{d!}\right)^{(r-1)/r}(1\wedge C^{1/r}s^{-l/r})+\frac{1}{d!}h_d(s).\end{align*}
Thus, $\sup_{n\in\N}U_n(s) s^{\a-1}$ is Lebesgue integrable over $[0,\infty)$. The dominated convergence theorem implies that $\int_0^{\infty}U_n(s)s^{\a-1}\,ds$ converges to $0$ as $n\to\infty$.
This completes the proof. 
\end{proof}

\subsection{The expression of the constant $I_{d-1}^{(\a)}$}\label{sec:I}
We now provide more concrete expressions for $I_{d-1}^{(\a)}$.
By an argument similar to that in \cite[Section~7.1]{HS}, $I_{d-1}^{(\a)}$ can be expressed as 
\[
I_{d-1}^{(\a)}=\frac1{d!\,(\a+1)}\left(\int_0^{t_d^*}\frac{(-\log s)^{\a+1}}{(1-s)^{d\a}}\,ds+(c_d^*)^\a\int_{t_d^*}^1 (-\log s)\,ds\right).
\]
For any $d\in\N$ and $\a>0$, 
\begin{align*}
\int_0^{t_d^*}\frac{(-\log s)^{\a+1}}{(1-s)^{d\a}}\,ds
&= \int_{-\log t_d^*}^\infty\frac{t^{\a+1}}{(1-e^{-t})^{d\a}}e^{-t}\,dt\\
&= \int_{-\log t_d^*}^\infty t^{\a+1}\sum_{k=0}^\infty\binom{d\a-1+k}k e^{-(k+1)t}\,dt\\
&= \sum_{k=0}^\infty\binom{d\a-1+k}k \int_{-\log t_d^*}^\infty t^{\a+1} e^{-(k+1)t}\,dt\\
&= \sum_{k=0}^\infty\binom{d\a-1+k}k \frac1{(k+1)^{\a+2}}\int_{-(k+1)\log t_d^*}^\infty u^{\a+1} e^{-u}\,du. 
\end{align*}
Here, we used the change of variable formulae with $t=-\log s$ in the first line and with $u=(k+1)t$ in the fourth line. 
For $d\in\N$ and $x>0$, define
\[
J_{d,k}(x)=\int_{-(k+1)\log t_d^*}^\infty u^{x-1} e^{-u}\,du. 
\]
Then, using integration by parts, 
\[
J_{d,k}(x+1)=x J_{d,k}(x)+(t_d^*)^{k+1}(k+1)^x(-\log t_d^*)^x. 
\]
As $J_{d,k}(1)=(t_d^*)^{k+1}$, we have
\[
J_{d,k}(x+1)=x!\,(t_d^*)^{k+1}\sum_{j=0}^x \frac{(k+1)^j(-\log t_d^*)^j}{j!}
\]
for $x\in\N\cup\{0\}$. Then, for any $d,\a\in\N$, 
\begin{align*}
\int_0^{t_d^*}\frac{(-\log s)^{\a+1}}{(1-s)^{d\a}}\,ds
&= \sum_{k=0}^\infty\binom{d\a-1+k}k \frac{J_{d,k}(\a+2)}{(k+1)^{\a+2}}\\
&= (\a+1)!\,\sum_{k=0}^\infty\binom{d\a-1+k}k \frac{(t_d^*)^{k+1}}{(k+1)^{\a+2}} \sum_{j=0}^{\a+1} \frac{(k+1)^j(-\log t_d^*)^j}{j!}\\
&= (\a+1)!\,\sum_{j=0}^{\a+1} \frac{(-\log t_d^*)^j}{j!}\sum_{k=0}^\infty\binom{d\a-1+k}k \frac{(t_d^*)^{k+1}}{(k+1)^{\a+2-j}}\\
&= \frac{(\a+1)!}{(d\a-1)!}\sum_{j=0}^{\a+1} \frac{(-\log t_d^*)^j}{j!}\sum_{i=0}^{d\a-1}{d\a-1\brack i}\sum_{k=0}^\infty \frac{(t_d^*)^{k+1}}{(k+1)^{\a+2-i-j}}\\
&= \frac{(\a+1)!}{(d\a-1)!}\sum_{i=0}^{d\a-1}{d\a-1\brack i} \sum_{j=0}^{\a+1} \frac{(-\log t_d^*)^j}{j!} \Li_{\a+2-i-j}(t_d^*). 
\end{align*}
Here, $n\brack k$ denotes Stirling numbers of the first kind, that is, the coefficients of the identity 
\[
x(x+1)\cdots(x+n-1)=\sum_{i=0}^n{n\brack i}x^i,
\]
where ${0\brack 0}=1$ by convention, and $\Li_s(x)$ is the polylogarithm 
\[
\Li_s(x)=\sum_{k=1}^\infty \frac{x^k}{k^s}\quad\text{($s\in\Z$, $0\le x\le1$). }
\]
Thus, for any $d,\a\in\N$, 
\begin{align}\label{eq:I1}
I_{d-1}^{(\a)}
&= \frac1{d!}\biggl\{\frac{\a!}{(d\a-1)!}\sum_{i=0}^{d\a-1}{d\a-1\brack i} \sum_{j=0}^{\a+1} \frac{(-\log t_d^*)^j}{j!} \Li_{\a+2-i-j}(t_d^*)\\
&\qad +\frac{(c_d^*)^\a\{-\log t_d^* -(1-t_d^*)\}}{d(\a+1)}\biggr\}, \nonumber
\end{align}
where we used the identity $(d+1)(1-t_d^*) + (1+d t_d^*)\log t_d^* = 0$ for the last term. 
In particular, noting that $t_1^*=1$, for any $\a\in\N$, 
\[
I_0^{(\a)}
= \a\sum_{i=0}^{\a-1}{\a-1\brack i}\Li_{\a+2-i}(1)
= \a\sum_{i=0}^{\a-1}{\a-1\brack i}\zeta(\a+2-i). 
\]
In particular, we have the specific values 
\begin{align*}
I_0=I_0^{(1)}&=\zeta(3), \\
I_0^{(2)}&=2\zeta(3), \\
I_0^{(3)}&=3(\zeta(3)+\zeta(4)), \\
I_0^{(4)}&=4(\zeta(3)+3\zeta(4)+2\zeta(5)), \\
I_0^{(5)}&=5(\zeta(3)+6\zeta(4)+11\zeta(5)+6\zeta(6)). 
\end{align*}
Also, by letting $\a=1$ and $d\ge2$ in \eqref{eq:I1}, we obtain
\begin{align}\label{eq:I2}
I_{d-1}
&=\frac1{d!}\biggl[\frac1{(d-1)!}\sum_{i=0}^{d-1}{d-1 \brack i}\left\{\Li_{3-i}(t_d^*)+(-\log t_d^*)\Li_{2-i}(t_d^*)+\frac{(-\log t_d^*)^2}{2}\Li_{1-i}(t_d^*)\right\}\\
&\qquad +\frac{(-\log t_d^*)\{-\log t_d^* -(1-t_d^*)\}}{2d(1-t_d^*)^d}\biggr].\nonumber
\end{align}
In particular,
\begin{align*}
I_1&=\frac12\biggl[\Li_2(t_2^*)+(\log t_2^*)\log(1-t_2^*)+\frac{t_2^*(\log t_2^*)^2}{2(1-t_2^*)}+\frac{(\log t_2^*)\{\log t_2^* +(1-t_2^*)\}}{4(1-t_2^*)^2}\biggr],\\
I_2&=\frac1{12}\biggl[\Li_2(t_3^*)+(\log t_3^*-1)\log(1-t_3^*)+\frac{t_3^*(\log t_3^*)(\log t_3^*-2)}{2(1-t_3^*)}+\frac{t_3^*(\log t_3^*)^2}{2(1-t_3^*)^2}\\
&\qquad+\frac{(\log t_3^*)\{\log t_3^* +(1-t_3^*)\}}{3(1-t_3^*)^3}\biggr].
\end{align*}
\section*{Acknowledgements}
This study was supported by JSPS KAKENHI Grant Number JP15H03625.
The authors thank Professors Yasuaki~Hiraoka and Tomoyuki~Shirai for their valuable comments.
%
%
%
%
%
%
%
%
%
%

\end{document}